\documentclass[twoside,draft,reqno,12pt]{amsart}

\usepackage[english]{babel}
\usepackage{amsmath}
\usepackage{amssymb}
\usepackage{enumerate}
\usepackage{esint}

\usepackage{pgf,tikz}
\usetikzlibrary{arrows}
\usepackage{caption}

\usepackage{amsfonts}
\usepackage{graphicx}
\usepackage{ mathrsfs }

\usepackage[ citecolor=black, linkcolor=black,
colorlinks, hypertexnames=false, plainpages=false]{hyperref}

\usepackage{color}

\usepackage[left=2cm,right=2cm,top=2cm,bottom=2cm]{geometry}

\newtheorem{theorem}{Theorem}[section]
\newtheorem{lemma}[theorem]{Lemma}

\newtheorem{definition}[theorem]{Definition}

\newtheorem{remark}[theorem]{Remark}

\numberwithin{equation}{section}

\newcommand{\1}{\text{\bf 1}}    
\newcommand{\R}{\mathbb R}

\newcommand{\dis}{\displaystyle}

\newcommand{\mmmintone}[1]{{\dis{\int\kern -.38cm
-}}_{\kern-.21cm\substack{#1}}\;\;}
\newcommand{\mmmintwo}[2]{{\dis{\int\kern -.43cm
-}}_{\kern-.21cm\substack{#1}}^{\substack{#2}}\;\;}
\newcommand{\submint}{{\scriptstyle{\int\kern -.66em -}}}
\newcommand{\submintone}[1]{{\scriptstyle{\int\kern -.66em
-}}_{\scriptscriptstyle{\kern-.21em\substack{#1}}}}
\newcommand{\fracmint}{{\textstyle{\int\kern -.88em -}}}
\newcommand{\fracmintone}[1]{{\textstyle{\int\kern -.88em
-}}_{\scriptscriptstyle{\kern-.21em\substack{#1}}}\;}

\newcommand{\eps}{\varepsilon}

\newcommand{\om}{\omega}

\newcommand{\E}{\mathbb E}

\newcommand{\nada}[1]{}

\def\be{\begin{equation}}
\def\ee{\end{equation}}

\newcommand{\T}{\mathbb{T}}
\newcommand{\Z}{\mathbb{Z}}

\def\nn{\textnormal{nn}}
\def\bad{\textnormal{bad}}
\def\good{\textnormal{good}}

\newcommand{\llav}[1]{  \left\{#1\right\} }
\newcommand{\pic}[1]{  \left\langle #1\right\rangle }
\newcommand{\norm}[1]{  \left\|#1\right\| }
\newcommand{\pare}[1]{  \left(#1\right) }
\newcommand{\corch}[1]{  \left[#1\right] }
\newcommand{\abs}[1]{  \left|#1\right| }

\usepackage{mathtools}
\DeclarePairedDelimiter\ceil{\lceil}{\rceil}

\let\originalleft\left
\let\originalright\right
\renewcommand{\left}{\mathopen{}\mathclose\bgroup\originalleft}
\renewcommand{\right}{\aftergroup\egroup\originalright}

\def\defi{\mathrel{\mathop:}=}
\def\fide{=\mathrel{\mathop:}}

\begin{document}
\today

\vskip.5cm

\title[Thermodynamics and Young-Gibbs measures]
{Thermodynamics for
spatially inhomogeneous magnetization and Young-Gibbs measures}

\author{Alessandro Montino}
\address{Alessandro Montino, Gran Sasso Science Institute\newline
\indent L' Aquila, Italy}
\email{alessandro.montino@gssi.infn.it}

\author{Nahuel Soprano-Loto}
\address{Nahuel Soprano-Loto, Gran Sasso Science Institute\newline
\indent L' Aquila, Italy}
\email{nahuel.sopranoloto@gssi.infn.it}

\author{Dimitrios Tsagkarogiannis}
\address{Dimitrios Tsagkarogiannis,
University of Sussex \newline
\indent Brighton, U.K.}
\email{D.Tsagkarogiannis@sussex.ac.uk}

\begin{abstract}
We derive thermodynamic functionals for
spatially inhomogeneous magnetization on a torus in the context of an Ising spin lattice model.
We calculate the corresponding free energy and pressure
(by applying an appropriate external field using a quadratic Kac potential)
and show that they are related via a modified Legendre transform.
The local properties of the infinite volume Gibbs measure, related to whether a macroscopic configuration is
realized as a homogeneous state or as a mixture of pure states, are also studied by constructing
the corresponding Young-Gibbs measures.
\end{abstract}

\maketitle


\section{Introduction}

In continuum mechanics, in order to describe the properties of a material, one studies a minimization
problem of a given free energy functional with respect to an appropriate order parameter.
The physical properties of the system are encoded in this functional which, in accordance with
the second law of thermodynamics, is a convex function.
Of particular interest is the case
when we are in the regime of phase transition between pure states, which corresponds to a linear
segment in the graph of the above functional with respect to the order parameter.
In such a case,
the solution of the minimization problem
can be realized as a fine mixture of the two pure phases of the system. 
This is the case of occurrence of {\it microstructures}, a phenomenon observed in materials with significant technological
implications.
The percentage of each phase in this mixture has been successfully 
described by the use of Young measures. 
For an overview, one can look at \cite{M99} and the references therein. 
On the other hand, from an atomistic viewpoint and at finite temperature, 
there is a well-developed rigorous theory of phase transitions. For example, in the case of the Ising model, each pure phase is described via an extremal Gibbs measure and mixtures via convex combinations of the extremal ones. 
In this paper, we connect the two descriptions and derive a macroscopic continuum mechanics theory for scalar order parameter starting from statistical mechanics.
In this context, we study the appearance of microstructures in our model by constructing {\it Young-Gibbs measures}, as they were introduced by Koteck\'y and Luckhaus in \cite{KL14} for the case of elasticity.
For more
analogues of Young measures in the analysis of the collective
behaviour in interacting particle systems,
see \cite{P12}.

To fix ideas, we consider the Ising model with nearest-neighbour ferromagnetic interaction as reference Hamiltonian. 
To allow spatially macroscopic inhomogeneous magnetization profiles, we have to patch together 
such
Ising models for each given macroscopic magnetization.
To obtain the desired profile, we can either do it by directly imposing a canonical constraint
or by adding an external magnetic field to the Hamiltonian. 
We follow the second strategy and implement it by using a Kac potential acting at an intermediate scale and
penalizing deviations out of an associated average magnetization
(in other words, fixing the magnetization in a weaker sense than the canonical constraint).
We study the Lebowitz-Penrose limit of the corresponding free energy and pressure and show equivalence of
ensembles. 
As a result, for every macroscopic magnetization, there is a unique external field that can produce it.
Note that this fact is not true for the nearest-neighbour Ising model in the phase transition regime at zero external field.
Indeed, thanks to the Kac term, we are able to fix a given value of the magnetization at large scales,
but this is still not possible at smaller ones.
In fact, what we observe in these smaller scales is the persistence of the two pure states of the Ising model with a percentage
determined by the overall macroscopic magnetization.

It is worth mentioning that, for the case of the canonical ensemble with homogeneous magnetization, 
the actual geometry of the location of the pure states has been investigated in the celebrated result of the construction of the Wulff shape for the Ising model, \cite{DKS92}. In an inhomogeneous set-up, the equivalent problem would be to further investigate
how such shapes corresponding to two neighbouring macroscopic points are connected, but this is a challenging question beyond the scope of
our paper.

To summarize, the presence of the Kac term in the Hamiltonian produces the phenomenon of microstructure
as a competition between the Ising factor, which prefers the spins aligned, and the long range averages, which tend to
keep the average fixed as
induced by the Kac term.
As a consequence, modulated patterns made out of the pure states are created and macroscopic values of the magnetization
are realized in this manner.
The percentage of each pure state in such a mixture is captured by the Young measure.
However, it would be desirable to study more detailed properties such as the
geometric shape of such structures. 
At zero temperature, there have been several studies at both the mesoscopic-macroscopic scale
(without claim of being exhaustive, we refer to \cite{M93, CO, ACO} for a rigorous analysis)
and the microscopic scale for lattice models, as in a recent series of works by Giuliani, Lebowitz and Lieb, see \cite{GLL09} and the references therein.
It would be of fundamental importance 
to develop such a theory in finite temperature 
as one would like to incorporate
fluctuation-driven phenomena. However, this is still beyond the available techniques.

The paper is organized as follows: in Section \ref{notation}, we present the model and the main theorems. The proof 
of the limiting free energy and pressure is given in Section \ref{sec3}.
This is a standard result that essentially follows after putting 
together the results for
the homogeneous case, which is also recalled in Appendix \ref{hom}.
In Section \ref{see}, we prove equivalence of ensembles. 
In Section \ref{sld}, as a corollary of the large deviations, 
we show that spin averages in domains larger than the Kac scale converge in probability to the fixed macroscopic configuration.
The second part of the paper deals with investigating what happens when we take such averages in domains smaller than the
Kac scale. 
We see that, in the phase transition regime, local averages converge in probability to averages with respect to a mixture of the pure states
which, in accordance to the theory of the deterministic case, we call Young-Gibbs measure.
The relevant proofs are given in Section \ref{sYG} with some details left for Appendix \ref{appB}. 

\section{Notation and  results}\label{notation}

Let $\T \defi  \left[-\frac{1}{2}, \frac{1}{2} \right)^d$ be the $d$-dimensional unit torus. 
For $q\in \mathbb N$, we consider a small scaling parameter $\eps$ of the form $2^{-q}$. In this case, 
$\lim_{\eps\rightarrow 0}$ stands for $\lim_{q\rightarrow \infty}$.
The microscopic version of ${\T}$ is the lattice $\Lambda_{\eps} \defi  \pare{\eps^{-1}{\T}}\cap{\mathbb Z}^d$.
For a non-empty subset $A\subset {\mathbb Z}^d$, let $\Omega_{A} \defi \llav{-1,1}^{A}$ be the set of configurations $\sigma$ in $A$ that gives the value of the spin $\sigma(x)\in\{-1,1\}$ in each lattice point $x\in A$. 
Whenever needed, we also use the notation $\sigma_A$.

Given a scalar function $\alpha\in{C}({\T},{\mathbb R})$, which plays the role of an inhomogeneous external field, we define the
Hamiltonian  $H_{\Lambda_{\eps},\gamma, \alpha}:\Omega_{\Lambda_{\eps}}\rightarrow {\mathbb R}$ as follows:\begin{align}\label{hamil}
H_{\Lambda_{\eps},\gamma, \alpha}(\sigma) \defi  H_{\Lambda_{\eps}}^{\nn}(\sigma)+
K_{\Lambda_{\eps},\gamma, \alpha}(\sigma).
\end{align}
The first part
$H_{\Lambda_\eps}^{\textnormal{nn}}:\Omega_{\Lambda_\eps}\rightarrow \R$ is defined by
\begin{align}\label{Hnn}
H_{\Lambda_\eps}^{\nn}(\sigma) \defi 
-\sum_{\substack{x,y\in\Lambda_{\eps} \\[0.05cm] x\sim y }}\sigma(x)\sigma(y),
\end{align}
where $x\sim y$ means that $x$ and $y$ are nearest-neighbour sites, assuming {\it periodic boundary conditions} in the box $\Lambda_\eps$.
The second part is
\begin{equation}\label{K}
K_{\Lambda_{\eps},\gamma,\alpha}(\sigma) \defi 
\sum_{x\in\Lambda_{\eps}}
(I^{\gamma}_x(\sigma)-
\alpha(\eps x))^2,
\end{equation}
where
\begin{align}\label{Igamma}
I^{\gamma}_x(\sigma) \defi \sum_{y\in\Lambda_\eps}J_{\gamma}(x,y)\sigma(y)
\end{align}
is an average of the configuration $\sigma$ around a vertex $x\in\Lambda_\eps$.
We introduce
another small parameter $\gamma>0$ 
and
the Kac interaction $J_{\gamma}:\Lambda_\eps\times \Lambda_\eps\rightarrow {\mathbb R}$ defined by
\begin{align}\label{JJJ}						
J_{\gamma}(x,y) \defi  \gamma^d\phi(\gamma(x-y)).
\end{align}
Here $\phi\in C^2({\mathbb R}^d, [0,\infty))$ is an even function
that vanishes outside the unit ball $\{ r\in\mathbb R^d:|r|< 1 \}$ 
and  integrates to $1$.
The difference $x-y$ appearing in the right-hand side of \eqref{JJJ} is a difference modulo $\Lambda_\eps$.
Hence, the second term enforces the averages of spin configurations to follow $\alpha$. 
Given \eqref{hamil}, the associated finite volume Gibbs measure is defined by
\begin{align}
\mu_{\Lambda_{\eps}, \gamma, \alpha}(\sigma) \defi  \frac{1}{Z_{\Lambda_{\eps},  \gamma, \alpha}}e^{-\beta H_{\Lambda_{\eps},\gamma, \alpha}(\sigma)},
\end{align}
where $Z_{\Lambda_{\eps},  \gamma, \alpha}$ is the normalizing constant.
Note that throughout this paper, we neglect  from the notation the dependence on $\beta$.

To study inhomogeneous magnetizations, we assume that locally in the macroscopic scale 
(i.e. the scale of the torus $\T$) we have obtained a given value of the
magnetization,
which can however vary slowly as we move from one point to another.
To describe what ``locally'' and ``varying slowly'' mean, we introduce an intermediate scale $l$ of the form $2^{-p}$, $p\in\mathbb N$.
Again, $\lim_{l\rightarrow 0}$ stands for $\lim_{p\rightarrow\infty}$.
Let $\{C_{l,1},\dots,C_{l,N_l}\}$ be the natural partition $\mathscr C_l$ of $\T$ into $N_l=l^{-d}$ cubes  of side-length $l$,
and let $\{\Delta_{\eps^{-1}l,1},\ldots,\Delta_{\eps^{-1}l,N_l}\}$ be its microscopic version, denoted by $\mathscr D_{\eps^{-1}l}$.
Its elements are given by
$\Delta_{\eps^{-1}l,i} \defi (\eps^{-1}C_{l,i})\cap \Z^d$ for every $i=1,\ldots,N_l$.
Making an abuse of notation, for every $i$, we identify the set $\Delta_{\eps^{-1}l,i}\subset \Z^d$ with the set $\cup_{x\in \Delta_{\eps^{-1}l,i}}
\Delta_{1}(x)$ in $\mathbb R^d$, where $\Delta_1(x)$ is the cube of size $1$ centered in $x$.
Note that $|\Delta_{\eps^{-1}l}|$ is the volume of the set $\Delta_{\eps^{-1}l}$, 
but also the cardinality of points in $\mathbb Z^{d}$ within the set $\Delta_{\eps^{-1}l}$.
Given $u\in{C}({\T},\pare{-1,1})$, let $u^{(l)}:\T\rightarrow(-1,1)$ be the piece-wise constant approximation of $u$ at scale $l$:
for all $r\in C_{l,i}$,
\begin{align}\label{ul}
u^{(l)}(r)=\bar u^{(l)}_i \defi \frac{1}{|C_{l}|}\int_{C_{l,i}}u(r')dr'.
\end{align}
Here $|C_{l}|=l^d$ denotes the volume of any of the cubes $C_{l,1},\ldots,C_{l,N_l}$.
For $A$ and $B$   non-empty subsets of ${\mathbb Z}^d$ such that $A\subset B$, and for $\sigma\in\Omega_B$, we define the average magnetization of $\sigma$ in $A$ by 
\begin{equation}\label{magn}
m_A(\sigma) \defi \frac{1}{\abs{A}}\sum_{x\in A}\sigma(x).
\end{equation}
For $n\in\mathbb N$, we define the set
\begin{align}
 I_n \defi \llav{\frac{2i-n}{n}:i\in{\mathbb Z}\cap\corch{0,n}}.
\end{align}
Observe that, under this definition, $I_{\abs{A}}$ is the set of all possible (discrete) values that $m_A$ can assume.
For $t\in [-1,1]$, let $\lceil t \rceil_n$ be the value in $I_{n}$ corresponding to the best approximation of 
$t$ from above: 
\begin{align}
\lceil t\rceil_n  \defi \min\llav{t'\in I_{n}:t'\geq t}.
\end{align}
Furthermore, we consider the set 
\begin{align}\label{set}
\Omega_{\Lambda_{\eps},l}(u) \defi \{\sigma\in\Omega_{\Lambda_{\eps}}:
m_{\Delta_{\eps^{-1}l,i}}(\sigma)=
\lceil u^{(l)}_i\rceil_{|\Delta_{\eps^{-1}l}|}, \ \ \forall i=1,\ldots,N_l\}
\end{align}
of all configurations whose locally averaged magnetization $m_{\Delta_{\eps^{-1}l,i}}$ is close to the average of
$u$ in the corresponding macroscopic
coarse grained box $C_{l,i}$ (see \eqref{ul}), for every $i=1\ldots,N_l$.
We have:

\begin{theorem}[Free energy and pressure]\label{ttfree}
For $u\in{C}({\T},\pare{-1,1})$ and $\alpha\in{C}({\T},{\mathbb R})$, we have
\begin{align}\label{free_energy}
\lim_{l \rightarrow 0}
\lim_{\gamma\to 0}
\lim_{\eps \rightarrow 0}-\frac{1}{\beta |\Lambda_{\eps}|} \log  {\sum_{\sigma \in \Omega_{\Lambda_\eps ,l}(u)} e^{-\beta H_{\Lambda_{\eps},\gamma, \alpha}(\sigma)}}=\int_{\T} \corch{ f_\beta (u(r)) +(u(r)-\alpha(r))^2}dr
\fide F_{\alpha}(u).
\end{align}
This limit gives the infinite volume free energy associated to the Hamiltonian \eqref{hamil}.
Here $f_\beta$ is the infinite volume free energy associated to the Hamiltonian \eqref{Hnn} (see Theorem \ref{tfree}).
Similarly, we obtain the infinite volume pressure
\begin{align}\label{pressure}
\lim_{\gamma \rightarrow 0} \lim_{\eps \rightarrow 0} \frac{1}{\beta |\Lambda_{\eps}| } \log Z_{\Lambda_{\eps},  \gamma, \alpha}=
-\min_{u \in {C}({\T},(-1,1))} \, F_{\alpha}(u)\fide P(\alpha).
\end{align}
Moreover, given $I_{ \alpha}: {C}({\T},\pare{-1,1})\rightarrow {\mathbb R}$ 
defined by
\begin{align}\label{I}
I_{\alpha}(u) \defi  F_{\alpha}(u)-\min_{v\in{C}({\T},(-1,1))}F_{\alpha}(v),
\end{align}
we obtain the following Large Deviations limit:
\begin{align}\label{LDlimit}
\lim_{l \rightarrow 0} \, \lim_{\gamma \rightarrow 0} \, \lim_{\eps \rightarrow 0} \, \frac{1}{\beta\abs{\Lambda_{\eps}}} \log \mu_{\Lambda_{\eps},  \gamma, \alpha}(\Omega_{\Lambda_\eps,l}(u))=-I_{\alpha}(u),
\end{align}
where the set $\Omega_{\Lambda_\eps,l}(u)$ is defined in \eqref{set}.
\end{theorem}

The proof is given in Section \ref{sec3}. 

\begin{remark}\label{fixing_u}
The minimization problem in
Theorem \eqref{ttfree} can be easily solved;
indeed, since $f_\beta$ is convex, symmetric with respect to the origin and $\lim_{t \rightarrow \pm 1} f_\beta'(t)=\pm \infty$,
the associated Euler-Lagrange equation
\begin{equation}\label{EL}
f_\beta'(u)+2(u-\alpha)=0
\end{equation}
has a unique solution
$u \defi \tilde u(\alpha)$ for every number $\alpha\in\mathbb R$.
On the other hand, for a given $u\in (-1,1)$, if we choose 
$\tilde\alpha(u) \defi u+\frac 12 f_\beta'(u)$, then we can say that
the Hamiltonian $H_{\Lambda_\eps,\gamma, \alpha}$ 
with $\alpha=\tilde\alpha(u)$ fixes the magnetization profile $u$ in the sense of large deviations.
The same is true point-wisely for functions, namely, $x\mapsto \tilde u(\alpha(x))$ is the minimizer of $F_{\alpha}$ in $C(\T,(-1,1))$.
\end{remark}

In Remark \ref{fixing_u}, we have established a 
relation between a fixed macroscopic magnetization $u$ and the way to obtain it by imposing an appropriate external field
$\tilde\alpha(u)$ via a grand canonical ensemble with Hamiltonian $H_{\Lambda_\eps,\gamma, \tilde \alpha (u)}$. 
There is, however, an important difference with respect to the case of the Ising model with homogeneous external magnetic field:
in the case of homogeneous magnetization, the correspondence between values of the external field and values of the 
magnetization is not one-to-one due to the fact that $f_\beta$ is constant on the interval $[-m_\beta,m_\beta]$ (to be specified later). 
On the contrary, in our model, we obtain such an one-to-one correspondence because of the presence of the Kac term 
which, acting at an intermediate scale,
assigns a value to the magnetization according to the external field.
This is manifested by a new quadratic term appearing in the free energy.

In the following theorem, 
we prove a duality relation between the free energy that corresponds to the Ising part of the Hamiltonian
\eqref{Hnn}
and the pressure $P(\alpha)$, obtained 
through a modified Legendre transform with the external field action given by \eqref{K}. 

\begin{theorem}[Equivalence of ensembles] \label{equivalence}
For $\alpha \in {C}\pare{{\T}, \mathbb{R}}$, the following identity holds:
\begin{align} \label{equivalence_1}
P(\alpha)=\max_{u \in C({\T},\pare{-1,1})} { -F_{\alpha}(u) }.
\end{align}
Conversely, for $u\in{C}\pare{{\T},\pare{-1,1}}$,
\begin{align} \label{equivalence_2}
\int_{\T} f_\beta(u(r))dr = \max_{\alpha \in {C}\pare{{\T},{\mathbb R}}} \bigg\{-P(\alpha)  -\int_{\T} (u(r)-\alpha(r))^2 dr\bigg\}.
\end{align}
\end{theorem}

\medskip

The proof is given in Section \ref{see}.
As we mentioned before, the Kac potential acts at an intermediate scale $\gamma^{-1}$ and tends to fix
the average of the spin values in any box larger than $\gamma^{-1}$.
To state this result properly, we recall the empirical magnetization defined in \eqref{magn} and, with a slight abuse of notation, we extend it to a function
from $\T$ to $[-1,1]$ given by
$r\mapsto m_{B_{R}(\eps^{-1}r)}$
in such a way that it is constant in each small cube of side-length $\eps$.
Here $B_R(x)$ is the ball of radius $R$ with center $x$, taking into consideration the periodicity in  $\Lambda_\eps$.
The first result asserts that, for $\alpha\in C(\T,\R)$ and $R_\gamma\gg \gamma^{-1}$, empirical averages 
converge in probability to the magnetization profile $u=\tilde u(\alpha)$.
Formally, we define the test operator $L_{\om,g}:L^1(\T,[-1,1])\rightarrow \R$, depending on a function $\om\in C(\T,\R)$ and on a Lipschitz function $g:[-1,1]\rightarrow \R$, by
\begin{align}\label{okm}
L_{\om,g}(u) \defi \int_{\T}\om(r)g(u(r))dr.
\end{align}
Under this definition, the following theorem asserts that the operator applied to the empirical average
\begin{equation}\label{okm1}
L_{\om,g}(m_{B_{R}(\eps^{-1}\cdot)})=
\eps^{d}\sum_{x\in\Lambda_{\eps}}\omega(\eps x)g(m_{B_{R}(x)}(\sigma))
\end{equation}
converges to $L_{\om,g}(u)$ 
in $\mu_{\Lambda_\eps,\gamma,\alpha}$-probability.
Note that this convergence is a bit different than the usual convergence in probability, 
since the measure $\mu_{\Lambda_\eps,\gamma,\alpha}$ changes as $\eps\to 0$.

\begin{theorem}\label{thmcor}
Let $u \in{C}\pare{{\T},(-1,1)}$
and choose $\alpha \defi \tilde\alpha(u)$ as in Remark \ref{fixing_u}.
Then, for $L_{\om,g}$ given in \eqref{okm}, $R_\gamma\gg \gamma^{-1}$
and $\delta>0$, we have
\begin{align}\label{inmeasure}
\lim_{\gamma\rightarrow 0}\lim_{\eps\rightarrow 0}\mu_{\Lambda_{\eps},  \gamma, \alpha}\Big(  |
L_{\om,g}(m_{B_R(\eps^{-1}\cdot)})-L_{\om,g}(u)
 |>\delta\bigg)=0.
\end{align}
\end{theorem}
The proof is given in Section \ref{sld}.
As it will be evident in the proof, in the above case $R_{\gamma}\gg \gamma^{-1}$,
the test function $g$ is not relevant.

A different scenario is observed when considering a smaller scale $R$: the value of
the random sequence $m_{B_{R}(x)}(\sigma)$ may  oscillate and, as a consequence, 
its limiting value may not be just the average.
In this case,
we study more detailed properties of
the underlying microscopic magnetizations.
We refer to these as the ``microscopic'' spin statistics of the measure 
$\mu_{\Lambda_{\eps}, \gamma, \alpha}$ (as opposed to the ``macroscopic''
statistics given by large deviations).
More precisely, we investigate how the
limiting value $u(r)$ in \eqref{inmeasure} is realized 
in intermediate scales:
as a homogeneous state or as a mixture of the pure states, and how one can retain such an information
in the limit.
This is reminiscent of the theory of Young measures as applied to describe
microstructure; see \cite{M99} for an overview.
In fact, in order to describe it in our case, we will construct the appropriate Young measure.

\begin{definition}[Young measure]\label{youngdef}
A Young measure 
is a map
\begin{align*}
\nu:\T &\to \mathcal P([-1,1]) \\
r & \mapsto \nu(r)
\end{align*}
such that, for every continuous function $g:[-1,1] \rightarrow \mathbb{R}$, the map $r \mapsto \langle \nu(r), g \rangle$ is measurable. Here 
$\mathcal P([-1,1])$ is the space of probability measures on $[-1,1]$,
and $\langle \nu(r),g \rangle$ indicates the expected value of $g$ with respect to the probability $\nu(r)$.
\end{definition}

To state the main result, we need to recall some background.
For an external field $h\in{\mathbb R}$, let $H^{\nn}_{\Lambda_\eps, h}:\Omega_{\Lambda_{\eps}}\rightarrow {\mathbb R}$ be the Hamiltonian defined by
\begin{equation}\label{guante}
H^{\nn}_{\Lambda_\eps, h}
(\sigma) \defi  H_{\Lambda_\eps}^{\nn}\pare{\sigma}-h\sum_{x\in\Lambda_{\eps}}\sigma(x),
\end{equation}
and let $\mu^{\nn}_{\Lambda_\eps,h}$ be the associated finite volume measure
\begin{align}\label{Ising}
\mu^{\nn}_{\Lambda_\eps,h}(\sigma) \defi \frac{1}{Z_{\Lambda_\eps,h}^{\text{nn}}}e^{-\beta H_{\Lambda_\eps,h}^{\text{nn}}(\sigma)}.
\end{align}
It is known that the set $\mathcal G(\beta,h)$ of 
infinite volume Gibbs measures associated to \eqref{guante}
is a non-empty, weakly compact, convex set of probability measures on $\Omega_{\mathbb Z^d}$.
More specifically, in $d=2$ and for any pair $(\beta,h)$,
the set $\mathcal G(\beta,h)$ is the convex hull of two extremal elements $\mu^{\nn}_{h,\pm}$,
the infinite volume limits of \eqref{Ising} with $\pm$ boundary conditions.
Any non-extremal Gibbs measure can be uniquely expressed as a convex combination of
these two elements: if $G\in\mathcal G(\beta,h)$, then there exists unique $\lambda_G\in [0,1]$ such that
\begin{equation}\label{decomp}
G=\lambda_G \mu^{\nn}_{h,+} +(1-\lambda_G) \mu^{\nn}_{h,-}.
\end{equation}

We define the magnetization at the origin as the expectation
\begin{align}
\varphi(h) \defi \int \sigma_0\mu_h^{\nn}(d\sigma).
\end{align}
The function $\varphi:\R\rightarrow (-1,1)$ is odd, strictly increasing, continuous in every point $h\neq 0$, and satisfies
\begin{align}
\lim_{h\rightarrow \pm\infty}\varphi(h)=\pm 1.
\end{align}
There exists a critical value $\beta_c>0$ such that the limit
\begin{align}\label{mbeta}
m_\beta \defi \lim_{h\downarrow 0}\varphi(h)
\end{align}
is positive if and only if $\beta>\beta_c$; it is the so-called spontaneous magnetization.
Note that it also coincides with the magnetization associated to $\mu^{\nn}_{0,+}$:
$m_\beta=\int\sigma_0 \mu^{\nn}_{0,+}(d\sigma)$.
For $\beta\leq\beta_c$, we have $m_\beta=0$. In this case, for every $m\in (-1,1)$, there exist a unique value $h=h(m)\in\R$ such that $\varphi(h)=m$.
If $m_\beta>0$, the same is true for values of the magnetization such that $|m|>m_\beta$.
But, how about if $|m|\le m_\beta$?
This has been investigated in \cite{G79}, where the {\it canonical} infinite volume Gibbs measure has been
constructed.
As every magnetization $u\in [-m_\beta,m_\beta]$ can be uniquely written as a convex combination 
\begin{equation}\label{conv}
u=\lambda_u m_\beta-(1-\lambda_u)m_\beta, 
\end{equation}
with $\lambda_u\in [0,1]$,
then $u$ is the magnetization associated to the probability
\begin{align}\label{p1p1}
\lambda_u\mu^{\nn}_{0,+}+(1-\lambda_u)\mu^{\nn}_{0,-}.
\end{align}
Hence, although ``macroscopically" one observes the value $u$ of the magnetization, in intermediate 
(still diverging) scales, one observes mixtures of the $m_\beta$ and $-m_\beta$ phases with a
frequency given by $\lambda_{u}$.
\bigskip

The purpose of the next theorem is to investigate the above fact for inhomogeneous magnetizations,
namely by ``imposing" a macroscopic profile $u(r)$
in a grand canonical
fashion, as it is described in Remark \ref{fixing_u}. 
For low enough temperature and for $|u(r)|<m_{\beta}$, 
at large scales (beyond $\gamma^{-1}$), the system with Hamiltonian \eqref{hamil} tends to fix $u(r)$ while,
at smaller ones, it allows (large) fluctuations once their average over areas of order $\gamma^{-1}$ is compatible with
$u(r)$. Indeed, the result states that, at boxes of scale up to $\gamma^{-1}$, one of the
two pure phases $\pm m_\beta$ is observed while, at scales larger than $\gamma^{-1}$ we see $u(r)$.
To capture this phenomenon, we use the observable $L_{\om,g}$ given by \eqref{okm}.
With a slight abuse of notation, we can also view $L_{\om,g}$ as acting over
Young measures $\nu(r)\in\mathcal P([-1,1])$ as follows:
\begin{align}\label{okm2}
L_{\omega,g}(\nu) \defi \int_\T \om(r)\langle \nu(r),g\rangle dr.
\end{align}

\begin{theorem}[Parametrization by Young measures]\label{mainthm}
Let $u\in C(\T,(-1,1))$ and $\alpha=\tilde \alpha(u)\in C(\mathbb T,\mathbb R)$ be its associated external field
(given by the solution of \eqref{EL}). We have the following cases:
\begin{itemize}
\item[] \textbf{Case $R_\gamma\gg \gamma^{-1}$.}
For every $\delta>0$,
\begin{align}\label{first}
\lim_{\gamma\rightarrow 0} \lim_{\eps\rightarrow 0}
\mu_{\Lambda_\eps,\gamma,\alpha}(|L_{\om,g}(m_{B_{R_\gamma}(\eps^{-1}\cdot)})-L_{\om,g}(\nu_u)|>\delta)=0,
\end{align}
where the functional $L_{\om,g}$ is defined in \eqref{okm} and \eqref{okm2}, and the Young measure is given by
$\nu_{u}(r) \defi  \delta_{u(r)}$ for every $r\in \T$. 
Here $\delta_{u(r)}$ is the Dirac measure concentrated in $u(r)$. 
\medskip
\item[] \textbf{Case $R=O(1)$.}
Suppose $d=2$ and let $\beta>\log \sqrt 5$.
Then, for every $\delta>0$,
\begin{align}\label{second}
\lim_{\gamma\rightarrow 0} \lim_{\eps\rightarrow 0}
\mu_{\Lambda_\eps,\gamma,\alpha}(|L_{\om,g}(m_{B_{R}(\eps^{-1}\cdot)})-L_{\om,g}(\nu_{u,R})|>\delta)=0.
\end{align}
Here, for $r\in\T$ and $E\subset [-1,1]$ a Borel subset, the Young measure $\nu_{u,R}$ is given by
\begin{align}\label{nu2}
\nu_{u,R}(r)(E) \defi 
\begin{cases}
\mu^{\nn}_{h(u(r))}[m_{B_R(0)}\in E] & \text{ if }|u(r)|>m_\beta
\\[0.2cm]
(\lambda_{u(r)}\mu^{\nn}_{0,+}+(1-\lambda_{u(r)})\mu^{\nn}_{0,-})[m_{B_R(0)}\in E]
& \text{ if }|u(r)|\leq m_\beta
\end{cases},
\end{align}
where $\lambda_{u(r)}$ and $h(u(r))$ are given in \eqref{conv} and the discussion preceding it, respectively.
\medskip
\item[] \textbf{Case $1\ll R \ll \gamma^{-1}$.} Under the same hypothesis of the previous item ($d=2$ and $\beta>\log\sqrt 5$), for every $\delta>0$,
\begin{align}\label{third}
\lim_{R\rightarrow \infty}\lim_{\gamma\rightarrow 0} \lim_{\eps\rightarrow 0}
\mu_{\Lambda_\eps,\gamma,\alpha}(|L_{\om,g}(m_{B_{R}(\eps^{-1}\cdot)})-L_{\om,g}(\nu_{u})|>\delta)=0,
\end{align}
where
\begin{align}\label{nu3}
\nu_u(r) \defi 
\begin{cases}
\delta_{u(r)}, & \text{ if }|u(r)|>m_\beta
\\[0.2cm]
\lambda_{u(r)}\delta_{m_\beta}+(1-\lambda_{u(r)})\delta_{-m_\beta},
& \text{ if }|u(r)|\leq m_\beta
\end{cases}.
\end{align}
\end{itemize}
\end{theorem}

\noindent
The case $R_\gamma \gg \gamma^{-1}$ is only a restatement of Theorem \ref{thmcor}.
The proof of the case $R=O(1)$ is given in Section \ref{sYG}.
The case $1\ll R \ll \gamma^{-1}$ follows as a corollary of the previous case
and it is briefly presented in Subsection \ref{subsec:third}.

\bigskip

\section{Proof of Theorem \ref{ttfree}}\label{sec3}

In this section we prove the limits \eqref{free_energy} and \eqref{pressure}.
Then, the limit in \eqref{LDlimit} is a direct consequence.

\subsection{Proof of \eqref{free_energy}}

We first prove it for $\alpha$ and $u$ constant and then for the general case.

\subsubsection{Constant $u$ and $\alpha$.}
For $u\in I_{\abs{\Lambda_\eps}}$, we introduce the finite volume free energy associated to the Hamiltonian \eqref{hamil} by
\begin{align}\label{finvol}
F_{\Lambda_{\eps}, \gamma, \alpha}\pare{u} \defi  -\frac{1}{\beta\abs{\Lambda_{\eps}}}\log{\sum_{\substack{ \sigma\in\Omega_{\Lambda_{\eps}} \\[0.05cm] m\pare{\sigma}=u }} e^{-\beta H_{\Lambda_{\eps},\gamma, \alpha}\pare{\sigma}} }.
\end{align}
For a generic $u\in (-1,1)$ we prove that
\begin{align}\label{tfc}
\lim_{\gamma\rightarrow 0}\lim_{\eps\rightarrow 0}F_{\Lambda_{\eps}, \gamma, \alpha}\pare{\ceil*{u}_{\Lambda_\eps}}
=f_\beta(u)+\pare{u-\alpha}^2.
\end{align}
We proceed in three steps:
we first show that the limit $\lim_{\eps\rightarrow 0}F_{\Lambda_{\eps}, \gamma, \alpha}\pare{\ceil*{u}_{\Lambda_\eps}}$ exists for every $\gamma$;
we continue with a coarse-graining approximation and
conclude establishing lower and upper bounds.

\medskip

{\underline{\it Step 1: existence of the limit $\lim_{\eps\rightarrow 0}F_{\Lambda_{\eps}, \gamma, \alpha}\pare{\ceil*{u}_{\Lambda_\eps}}$ for fixed $\gamma>0$.}}
Since $\eps=2^{-q}$, with a slight abuse of notation we denote the volume by $\Lambda_q$ in order
to keep track of the dependence on $q$.
We have $|\Lambda_{q+1}|=2^d |\Lambda_q|$.
We also define the sequence of magnetizations $u_q \defi \ceil*{u}_{\Lambda_q}$.
It suffices to prove that the sequence $(F_{\Lambda_q, \gamma,\alpha}(u_q))_q$ is bounded below and that the inequality
\begin{align}\label{tt4}
F_{\Lambda_{q+1}, \gamma,\alpha}(u_{q+1})\leq F_{\Lambda_{q},\gamma,\alpha}(u_{q})
+a_q
\end{align}
holds for every $q$, where $(a_q)_q$ is 
a sequence of non-negative numbers such that $\sum_q a_q<\infty$.

The fact that the sequence $(F_{\Lambda_q, \gamma,\alpha}(u_q))_q$ is bounded from below follows from the inequalities
\begin{align}\label{tt5}
\frac{1}{\beta\abs{\Lambda_q}}\log\sum_{\substack{\sigma\in\Omega_{\Lambda_q} \\ m_{\Lambda_q}(\sigma)=u_q}}e^{-\beta H_{\Lambda_q,\gamma,\alpha}(\sigma)}
\leq \frac{1}{\beta\abs{\Lambda_q}}\log\sum_{\substack{\sigma\in\Omega_{\Lambda_q} \\ m_{\Lambda_q}(\sigma)=u_q}}e^{-\beta H_{\Lambda_q}^{\nn}(\sigma)}
\leq \frac{1}{\beta\abs{\Lambda_q}}\log\sum_{\sigma\in\Omega_{\Lambda_q} }e^{-\beta H_{\Lambda_q}^{\nn}(\sigma)}
\end{align}
and the fact that the right hand side of \eqref{tt5} converges to the pressure with zero external field;
see Theorem \ref{tfree}.
To show \eqref{tt4}, we write:
\begin{eqnarray}
F_{\Lambda_{q+1}, \gamma,\alpha}(u_{q+1})-F_{\Lambda_{q},\gamma,\alpha}(u_q)
& = &
\corch{F_{\Lambda_{q+1}, \gamma,\alpha}(u_{q+1})-F_{\Lambda_{q+1}, \gamma,\alpha}(u_q)}\nonumber\\
&&
+\corch{F_{\Lambda_{q+1}, \gamma,\alpha}(u_q)-F_{\Lambda_{q}, \gamma,\alpha}(u_q)}.
\end{eqnarray}
To find an upper bound for 
\begin{align}
F_{\Lambda_{q+1}, \gamma,\alpha}(u_q)-F_{\Lambda_{q},\gamma,\alpha}(u_q),
\end{align}
we use the same sub-additive argument leading to \eqref{salmoncita} in the proof of Theorem \ref{tfree}.
Indeed, repeating the argument appearing there, it can be proved that
\begin{align}
F_{\Lambda_{q+1}, \gamma,\alpha}(u_q)-F_{\Lambda_{q},\gamma,\alpha}(u_q)\leq C \gamma^{-d} 2^{-q},
\end{align}
where $C$ is independent of $q$.

On the other hand, to estimate
\begin{align}\label{tt6}
F_{\Lambda_{q+1}, \gamma,\alpha}(u_{q+1})-F_{\Lambda_{q+1}, \gamma,\alpha}(u_q),
\end{align}
we use the following continuity lemma whose proof is also given in Subsection \ref{add}.

\begin{lemma}\label{AA}
If $t$ and $t'$ are consecutive elements of $I_{|\Lambda_q|}$, then
\begin{align}\label{BB}
\abs{F_{\Lambda_q, \gamma,\alpha}(t)
-F_{\Lambda_q, \gamma,\alpha}(t')}
 \leq
 C 2^{-q}\gamma^{-d}
 +\frac{\log\abs{\Lambda_q}}{\abs{\Lambda_q}},
\end{align}
where $C$ is a constant that depends only on the dimension $d$.
\end{lemma}

\medskip

\noindent
The upper bound
\begin{align}
F_{\Lambda_{q+1}, \gamma,\alpha}(u_{q+1})-F_{\Lambda_{q+1}, \gamma,\alpha}(u_q)\leq
2^d\pare{C 2^{-q}\gamma^{-d}
 +\frac{\log\abs{\Lambda_{q+1}}}{\abs{\Lambda_{q+1}}}}
\end{align}
follows after using this lemma repeatedly:
indeed, $u_{q+1}$ can be obtained from $u_q$ moving through consecutive elements of $I_{|\Lambda_{q+1}|}$ in at most $2^d$ steps.
To conclude, define
\begin{align}
a_q \defi 2^d\pare{C 2^{-q}\gamma^{-d}
 +\frac{\log\abs{\Lambda_{q+1}}}{\abs{\Lambda_{q+1}}}}+O(2^{-q}\gamma^{-d})
\end{align}
and observe that $\sum_q a_q<\infty$.

\bigskip

{\underline{\it Step 2: approximation by coarse-graining.}}
We consider a microscopic parameter $L_\gamma$ of the form $2^{m}$, $m\in{\mathbb Z}^+$
depending on $\gamma$ such that $\gamma L_\gamma\to 0$ as $\gamma\to 0$.
In the sequel, in order to simplify notation we drop the dependence on $\gamma$ from the scale $L$.
Recall that by ${\mathscr C}_{\eps L}=\llav{C_{\eps L,i}}_{i}$ (respectively $\mathscr{D}_L=\llav{\Delta_{ L,i}}_{i}$),
we denote a macroscopic (respectively microscopic) partition of $\Lambda_{\eps}$
consisting
of $N_{\eps L} \defi (\eps^{-1}/ L)^d$ many elements.

We define a new coarse-grained interaction $J_\gamma^{(L)}$ on the new scale $L$.
Let $\Delta_{L,k}$, $\Delta_{L,k'} \in \mathscr D_L$; then
for every $x\in \Delta_{L,k}$ and $y\in \Delta_{L,k'}$ we define
\begin{equation}\label{JL}
J^{(L)}_\gamma(x,y) \defi \frac{1}{|\Delta_{L}|^2}\int_{\Delta_{L,k}\times \Delta_{L,k'}}\gamma^d \phi(\gamma|r-r'|)dr\,dr'.
\end{equation}
As before, $\abs{\Delta_{L}}$ denotes the cardinality of a generic box $\Delta_{L,i}$.
Since it assumes constant values for all $x\in \Delta_{L,k}$ and $y\in \Delta_{L,k'}$
we also introduce the notation 
\begin{equation}\label{Jbar}
\bar J^{(L)}_\gamma(k,k') \defi L^d J_\gamma^{(L)}(x,y).
\end{equation}
Note that, for any $k$, we have
\begin{equation}\label{normalized}
\sum_{k'} \bar J^{(L)}_\gamma(k,k')=\frac{1}{L^d}\int_{\Delta_{L,k}} dr \int_{\mathbb R^d} dr' \gamma^d \phi(\gamma |r-r'|)=1.
\end{equation}
Comparing to $J_\gamma$, we have the error
\begin{equation}\label{Kacerror}
\Big|
J_\gamma(x,y)-J^{(L)}_\gamma(x,y)
\Big|
\leq C \gamma^d (\gamma L)  \mathbf 1_{|x-y|\leq 2\gamma^{-1}},
\end{equation}
where the constant $C$ depends on $d$ and $\|D\phi\|_\infty$ (the sup norm of the first derivative of $\phi$).
For a macroscopic parameter $l$ (to be chosen $\eps L$ in this case) and for
$r\in C_{l,k}$, we define the piece-wise constant approximation of $\alpha$ at scale $l$ as in \eqref{ul}:
\begin{equation}\label{bara}
\alpha^{(l)}(r) \defi \sum_k\mathbf 1_{C_{l,k}}(r) \bar\alpha^{(l)}_k,
\quad \text{where} \quad
\bar\alpha^{(l)}_k \defi 
\frac{1}{|C_{l,k}|}\int_{C_{l,k}}\alpha(r')dr' .
\end{equation}
With this definition,
\eqref{Kacerror} implies that
\begin{equation}\label{Kacerror1}
\sup_{x\in \Delta_{L,k}}
\sup_{\sigma}
\Big|
\Big( \sum_y  J_\gamma(x,y)\sigma(y)-\alpha(\eps x)\Big)^2
- 
\Big(\sum_{y} J_\gamma^{(L)}(x,y)\sigma(y)-\alpha^{(\eps L)}(\eps x)\Big)^2
\Big|
\leq C \gamma L+\eps L.
\end{equation}
Note that using the notation $\bar J_\gamma^{(L)}$ and $\bar\alpha^{(\eps L)}_k$ we can write
\begin{equation}\label{cg}
\sum_{x\in \Lambda_\eps}
\Big(\sum_{y\in\Lambda_\eps} J_\gamma^{(L)}(x,y)\sigma(y)-\alpha^{(\eps L)}(\eps x)\Big)^2
=
L^d\sum_k
\Big(\sum_{k'} \bar J_\gamma^{(L)}(k,k') m_{\Delta_{L,k'}}(\sigma)-\bar\alpha^{(\eps L)}_k\Big)^2.
\end{equation}

For the nearest-neighbour part of the Hamiltonian, there are 
$|\partial C_{l}| N_{\eps L}=L^{d-1}(\eps^{-1}/L)^d=L^{-1}|\Lambda_\eps|$ nearest neighbours 
between the boxes $\Delta_{L,1},\ldots, \Delta_{L,N_{\eps L}}$; hence we have
\begin{align}\label{pincha}
H_{\Lambda_{\eps}}^{\nn}\pare{\sigma}=\sum_{k=1}^{N_{\eps L}}
H^{\nn}_{\Delta_{L,k}}\pare{\sigma}+O\pare{L^{-1}\abs{\Lambda_\eps}},
\end{align}
where $H_{\Delta_{k,i}}^{\nn}$ is considered with periodicity in the box $\Delta_{L,i}$.
Thus, to calculate \eqref{finvol}, we sum over all possible values $u_1, \ldots, u_{N_{\eps L}}$
of the magnetization in the boxes $\Delta_{L,k}$, with $k=1,\ldots, N_{\eps L}$, and obtain
\begin{align}\label{mainformula}
-\frac{1}{\beta\abs{\Lambda_\eps}}
\log
\sum_{\substack{ u_1,...,u_N \in I_{L^d} \\[0.05cm] \frac{1}{N} \sum_{k} u_k=\ceil*{u}_{\abs{\Lambda_\eps}}}}
\prod_{k=1}^N \sum_{\substack{\sigma_k \in\Omega_{\Delta_{L,k}} \\[0.05cm] m_{\Delta_{L,k}}\pare{\sigma_k}=u_k }}  e^{-\beta H_{\Delta_{L,k}}^{\nn}\pare{\sigma_k}} 
\exp\llav{-\beta L^d\sum_{k=1}^N\Big(\sum_{k'=1}^N  \bar J_{\gamma}^{(L)}(k,k') u_{k'}-\bar\alpha^{(\eps L)}_k\Big)^2}
\end{align}
with a vanishing error of the order
$|\Lambda_\eps| (\gamma L +\eps L + L^{-1})$, as follows from \eqref{Kacerror1}, \eqref{cg} and \eqref{pincha}.
Like before, we are using the simplified notation $N=N_{\eps L}$.

In the Appendix, Theorem \ref{tfree}, 
we prove that the convergence to the free energy $f_\beta$ is uniform, hence
the sum
\begin{align}
\sum_{\substack{\sigma_k \in\Omega_{\Delta_{L,k}} \\[0.05cm] m_{\Delta_{L,k}}\pare{\sigma_k}=u_k }}  e^{-\beta H_{\Delta_{L,k}}^{\nn}\pare{\sigma_k}} \end{align}
can be approximated by $e^{-\beta L^d f_\beta\pare{u_k}}$ with an error bounded by $e^{\beta L^d s\pare{L}}$, with 
$s \pare{L}\to 0$ as $L\to\infty$. 
Then, the overall error is also vanishing:
\begin{equation}
-\frac{1}{\beta\abs{\Lambda_\eps}}
N_{\eps L}
L^d s(L)\sim s(L)\to 0
\end{equation}
Finally, we are left with
\begin{align}\label{gatito}
-\frac{1}{\beta\abs{\Lambda_\eps}}
\log
\sum_{\substack{ u_1,...,u_N \in I_{L^d} \\[0.05cm] \frac{1}{N} \sum_{k=1}^N u_k=\ceil*{u}_{\abs{\Lambda_\eps}}}}
\pare{\prod_{k=1}^N e^{-\beta L^d f_\beta\pare{u_k}} } 
\exp\llav{-\beta L^d\sum_{k=1}^N\pare{\sum_{k'=1}^N  \bar J_{\gamma}^{(L)}(k,k') u_{k'}-\bar\alpha^{(\eps L)}_k}^2}.
\end{align}

{\underline{\it Step 3: upper and lower bounds.}}
To obtain a lower bound of \eqref{gatito}, we bound the sum in \eqref{gatito} by
the maximum contribution.
Note that the cardinality of the sum
vanishes in the limit $\eps\to 0$ after taking the logarithm and dividing by $\beta\abs{\Lambda_\eps}$.
Then the problem reduces to studying the minimum
\begin{align}\label{terito}
\min_{\substack{ u_1,...,u_N \in I_{L^d} \\[0.05cm] \frac{1}{N} \sum_{i=1}^N u_i=\ceil*{u}_{\abs{\Lambda_\eps}}}}
G\pare{u_1,\ldots,u_N},
\end{align}
where $G:\corch{-1,1}^N\rightarrow {\mathbb R}$ is the function defined by
\begin{align}
G\pare{u_1,\ldots,u_N} \defi \frac{1}{N} \sum_{i=1}^{N}f_\beta\pare{u_i}+ \frac{1}{N}\sum_{i=1}^N\pare{\sum_{j=1}^N \bar J^{(L)}_{\gamma}(i,j) u_j-\bar\alpha_i^{(\eps L)}}^2.
\end{align}
Moreover,
using the convexity of $f_{\beta}$, we have
\begin{align}\label{cortina1}
\frac{1}{N} \sum_{i=1}^N f_{\beta}(u_i)\geq 
f_{\beta} \pare{ \frac{1}{N} \sum_{i=1}^N u_i }.
\end{align}
Furthermore, from the convexity of the function $t\mapsto \pare{t-\bar\alpha^{(\eps L)}_i}^2$, using \eqref{normalized}, we obtain
\begin{align}\label{cortina2}
\frac{1}{N}\sum_{i=1}^N\pare{\sum_{j=1}^N \bar J^{(L)}_{\gamma}\pare{i,j}u_j-\bar\alpha^{(\eps L)}_i}^2
\geq \pare{\frac{1}{N}\sum_{i=1}^N \sum_{j=1}^N \bar J^{(L)}_{\gamma}(i,j) u_j -\bar\alpha^{(\eps L)}_i}^2 
=  \pare{\frac{1}{N}\sum_{j=1}^N u_j -\bar\alpha^{(\eps L)}_i}^2.
\end{align}
Thus, using \eqref{cortina1} and \eqref{cortina2}, 
expression \eqref{terito} can be bounded from below by
\begin{align}
f_\beta\pare{\ceil*{u}_{\abs{\Lambda_\eps}}}+(\ceil*{u}_{\abs{\Lambda_\eps}}-\bar\alpha^{(\eps L)}_i)^2,
\end{align}
which converges to $f_\beta\pare{u}+\pare{u-\alpha}^2$, and the lower bound follows.

For the upper bound of \eqref{gatito},
we take one particular element $\tilde{u}_1,...,\tilde{u}_N$ that realizes the value of the lower bound. 
In this way, we obtain a lower bound for the sum over all possible values of $u$ in \eqref{gatito}
that leads to the desired upper bound.
The idea is that these values should be as close as possible to $\ceil*{u}_{\abs{\Lambda_\eps}}$ and satisfy 
\begin{align}\label{sillita}
\frac{1}{N} \sum_{i=1}^N \tilde{u}_i=\ceil*{u}_{\abs{\Lambda_\eps}}.
\end{align}
Let $u^-$ and $u^+$ be the best possible approximations of $\ceil*{u}_{\abs{\Lambda_\eps}}$ in $I_{L^d}$ 
from below and from above, respectively. We have:
\begin{align}
u^- \defi \max\llav{t\in I_{L^d}:t\leq u} \hspace{1cm}
u^+ \defi \min\llav{t\in I_{L^d}:t\geq u}.
\end{align}
Notice that $u^+-u^-\leq \frac{2}{L^d}$.
We define
\begin{align}
\tilde{u}_i \defi 
\begin{cases}
&u^+, \text{ if } i=1 \\
&u^-, \text{ if } i \in \{2,...,N-1\} \text{ and } \frac{1}{i-1} \sum_{j=1}^{i-1} \tilde{u}_j > \ceil*{u}_{\abs{\Lambda_\eps}} \\
&u^+, \text{ if } i \in \{2,...,N-1\} \text{ and } \frac{1}{i-1} \sum_{j=1}^{i-1} \tilde{u}_j \leq \ceil*{u}_{\abs{\Lambda_\eps}} \\
&N \ceil*{u}_{\abs{\Lambda_\eps}} -\sum_{j=1}^{N-1} \tilde{u}_j, \text{ if } i=N 
\end{cases}.
\end{align}
Notice that identity \eqref{sillita} is satisfied by construction; moreover, it holds that
\begin{align}
\abs{\tilde{u}_i-\ceil*{u}_{\abs{\Lambda_\eps}}} \leq 2L^{-d} \ \ \  \forall i \in \{1,...,N \} .
\end{align}

As $u\in\pare{-1,1}$,
we can chose $\corch{a,b}\subset\pare{-1,1}$ such that $\ceil*{u}_{\abs{\Lambda_\eps}}\in\corch{a,b}$ and $\tilde u_i\in\corch{a,b}$ for every $i$, with $\eps$ small enough and $L$ large enough.
As $f'_\beta$ is bounded in $\corch{a,b}$ (see Theorem \ref{tfree}) and the function $t\mapsto t^2$ is Lipschitz over bounded subsets of ${\mathbb R}$, it follows that
\begin{align}
G\pare{\tilde u_1,\ldots,\tilde u_N}=G\pare{\ceil*{u}_{\abs{\Lambda_\eps}},\ldots,\ceil*{u}_{\abs{\Lambda_\eps}}}+O\pare{L^{-2}}.
\end{align}
Moreover, since
\begin{align}
\lim_{\eps \rightarrow 0} G\pare{\ceil*{u}_{\abs{\Lambda_\eps}},...,\ceil*{u}_{\abs{\Lambda_\eps}}}=f_\beta(u)+(u-\alpha)^2 ,
\end{align}
we conclude the proof of the upper bound and with that the proof of \eqref{tfc}.

\subsubsection{General $u$ and $\alpha$.}
For a macroscopic scale $l$ of the form $2^{-p}$, $p\in\mathbb N$, recall the
macroscopic partition ${\mathscr C}_l$ 
of $\T$ and let $\mathscr D_{\eps^{-1}l}$ be its microscopic version,
both with $N=l^{-d}$ elements.
Given the function $u\in C(\T, (-1,1))$, we define
\begin{equation}\label{baru}
 \bar U^{(l)}_i \defi \ceil*{ \bar u^{(l)}_i }_{\abs{\Delta_{\eps^{-1}l,i}}},\quad \text{where}\quad \bar u^{(l)}_i \defi \fint_{C_{l,i}} u(r)dr
\end{equation}
and $C_{l,i}$ is the macroscopic version of $\Delta_{\eps^{-1}l,i}$.
Note that the average  $\bar u^{(l)}_i$ does not depend on $\eps$, while the upper bound $ \bar U^{(l)}_i$ it does
due to the given discretization accuracy.
Similarly, for $\alpha\in C(\T,\mathbb R)$, we 
consider its coarse-grained version
$\alpha^{(l)}$ as in \eqref{bara}.
We next
apply the previous result (for constant values of $u$ and $\alpha$)
and pass to the limit $l\to 0$.
To implement this procedure, we approximate the Hamiltonian \eqref{hamil} 
by the sum of Hamiltonians over the boxes of the partitions with periodic boundary conditions.
Neglecting the interactions between neighbouring boxes, for the Ising part of the Hamiltonian,
we break $N_{l}\cdot|\partial \Delta_{\eps^{-1}l}|\sim |\Lambda_\eps| \eps l^{-1}$ many interactions.
Hence
\begin{align}
H^{\nn}_{\Lambda_\eps}\pare{\sigma}
=\sum_{i=1}^N H_{\Delta_{\eps^{-1}l,i}}^{\nn}(\sigma_i)
+O\pare{|\Lambda_\eps|\eps l^{-1}},
\end{align}
where $\sigma=\sigma_1\vee\ldots\vee\sigma_{N}$
and by $\vee$ we denote the concatenation on the sub-domains $\Delta_{\eps^{-1}l,i}$, $i=1,\ldots, N$.
Similarly, for the Kac interaction, we neglect
$O\pare{|\Lambda_\eps| \gamma^{-1}\eps l^{-1}}$ interactions and obtain
\begin{align}
\sum_{x \in \Lambda_\eps} \left(\sum_{y \in \Lambda_\eps}J_{\gamma}(x,y) \sigma(y)-\alpha_l\pare{\eps x}\right)^2=
\sum_{i=1}^N \, \sum_{x \in \Delta_{\eps^{-1}l,i}}\left(\sum_{y \in \Delta_{\eps^{-1}l,i}} J_{\gamma}(x,y)\sigma(y)-\alpha_l\pare{\eps x}\right)^2+
O\pare{|\Lambda_\eps|\gamma^{-1}\eps l^{-1}}.
\end{align}
Recalling the definition \eqref{set} of the set $\Omega_{\Lambda_\eps,l}$, we have

\begin{eqnarray}
\sum_{\sigma \in \Omega_{\Lambda_\eps,l}(u)} e^{-\beta H_{\Lambda_{\eps},\gamma, \alpha^{(l)}}(\sigma)}
& = & 
e^{O\pare{|\Lambda_\eps|\eps l^{-1}}
+
O\pare{|\Lambda_\eps|\gamma^{-1}\eps l^{-1}}
}
\sum_{\sigma \in \Omega_{\Lambda_\eps,l}(u)} \, \prod_{i=1}^N 
e^{- \beta H_{\Delta_{l,i}, \gamma, \bar\alpha^{(l)}_i}\pare{\sigma_{i}}}
\nonumber\\
& = &
e^{O\pare{|\Lambda_\eps|\eps l^{-1}}
+
O\pare{|\Lambda_\eps|\gamma^{-1}\eps l^{-1}}
}
\prod_{i=1}^N \, \sum_{\substack{ \sigma\in\Omega_{\Delta_{\eps^{-1}l,i}} \\[0.05cm] m_{\Delta_{\eps^{-1}l,i}}\pare{\sigma}=\bar u_i }}e^{-\beta H_{\Delta_{l,i},\gamma, \bar\alpha^{(l)}_i}(\sigma_i)}.
\end{eqnarray}
After applying the previous result for $u$ and $\alpha$ constant to each one of
the Hamiltonians $H_{\Delta_{l,i}, \gamma, \bar\alpha^{(l)}_i}$ (recall the definition \eqref{hamil}),
 taking the $\log$, dividing by $-\beta\abs{\Lambda_\eps}$ and passing to the limits $\lim_{\gamma\rightarrow 0}\lim_{\eps\rightarrow 0}$, we obtain
\begin{align}
\sum_{i=1}^N\abs{C_{l,i}}\corch{f_\beta\pare{ \bar U^{(l)}_i}+\pare{ \bar U^{(l)}_i-\bar \alpha^{(l)}_i}^2}.
\end{align}
Take finally $\lim_{l\rightarrow 0}$ to obtain $\int_{\T} \corch{f_\beta\pare{u(r)}+\pare{u(r)-\alpha(r)}^2}dr$
and complete the proof of \eqref{free_energy}.

\subsection{Proof of \eqref{pressure}}

This is similar to the previous proof.
For the case of $\alpha$ constant,
the existence of the limit $\eps\to 0$ for fixed $\gamma$ can be proved by the same sub-additivity argument
as before, without however the extra effort to keep the canonical constraint in the sequence of
boxes of increasing size.
Then, by the same coarse-graining argument, similarly to \eqref{gatito} we obtain
\begin{align}\label{b1}
\frac{1}{\beta\abs{\Lambda_\eps}}
\log
\sum_{ u_1,...,u_N \in I_{L^d} }
\pare{\prod_{i=1}^N e^{-\beta L^d f_\beta\pare{u_i}} } 
\exp\llav{-\beta L^d\sum_{i=1}^N\pare{\sum_{j=1}^N \bar J^{(L)}_{\gamma}(i,j) u_j-\alpha}^2}.
\end{align}
For an {\it upper bound}, given \eqref{cortina1} and \eqref{cortina2}, 
we take the maximum of all choices of $u_1, \ldots, u_N$ and bound \eqref{b1} by
\begin{align}
\max_{u_1,\ldots,u_N\in I_{L^d}}\llav{-f_\beta\pare{\frac{1}{N}\sum_{i=1}^Nu_i}-\pare{\frac{1}{N}\sum_{i=1}^Nu_i-\alpha}^2}.
\end{align}
The later quantity is further bounded by
\begin{align}
-\min_{u\in\corch{-1,1}}\llav{f_\beta\pare{u}+\pare{u-\alpha}^2}
\end{align}
and the upper bound follows.

For a {\it lower bound} we take one element (when all are equal) and obtain:
\begin{equation}
\liminf_{\gamma \rightarrow 0} \, \lim_{\eps \rightarrow 0} \, \frac{1}{\beta |\Lambda_\eps|} \log  \sum_{\sigma \in \Omega_{\Lambda_\eps}} e^{-\beta H_{\Lambda_{\eps},\gamma, \alpha}(\sigma)} \geq -\min_{u \in [-1,1]} \llav{ f_\beta(u)+(u-\alpha)^2 }.
\end{equation}

For a {\it general $\alpha\in C(\T,\mathbb R)$}, we consider the partition $\mathscr C_{\eps^{-1}l}$
of $N_l$ many elements and, in each box, we apply the previous result. We have
\begin{eqnarray}\label{b2}
&&
\lim_{\gamma \rightarrow 0} \, \lim_{\eps \rightarrow 0} \, 
\frac{1}{\beta\abs{\Lambda_\eps}}\log\prod_{i=1}^{N_l}\sum_{\sigma\in\Omega_{\Delta_{\eps^{-1}l,i}}}
e^{-\beta H_{\Delta_{\eps^{-1}l,i}, \gamma, \bar\alpha^{(l)}_i}(\sigma_i)}  \nonumber\\
&=&
\frac{1}{N_l}
\sum_{i=1}^{N_l}
\lim_{\gamma \rightarrow 0} \, \lim_{\eps \rightarrow 0} \,
\frac{1}{\beta\abs{\Delta_{\eps^{-1}l,i}}}\log \sum_{\sigma_i\in\Omega_{\Delta_{\eps^{-1}l,i}}}
e^{-\beta H_{ \Delta_{\eps^{-1}l,i}, \gamma, \bar\alpha^{(l)}_i}\pare{\sigma_i}} \nonumber \\
& = & -
\sum_{i=1}^{N_l} \abs{C_{l,i}}
\min_{u\in\pare{-1,1}}\llav{f_\beta\pare{u}+\pare{u-\bar\alpha^{(l)}_i}^2}.
\end{eqnarray}
Note that, for every $\alpha\in\mathbb R$, since $f_\beta$ is convex and $u\mapsto (u-\alpha)^2$ is strictly convex, 
the function $u\mapsto f_\beta(u)+(u-\alpha)^2$ is strictly convex, so its derivative is strictly increasing.
Hence equation
\begin{align}\label{zorro}
\frac{1}{2}f'_\beta\pare{u}+u=\alpha
\end{align} 
has only one solution $\tilde u(\alpha)$.
Since $\alpha\mapsto \tilde u(\alpha)$ is a continuous function,
taking the limit $l\to 0$ in \eqref{b2}, by the Dominated Convergence Theorem we obtain
\begin{align}
\int_{\T}dr \corch{f_\beta\pare{\tilde u\pare{\alpha(r)}}+\pare{\tilde u\pare{\alpha(r)}-\alpha(r)}^2}.
\end{align}
Since this coincides with
\begin{align}
-\min_{v \in {C}({\T},[-1,1])} \, \int_{\T} dr \corch{f_\beta(v(r))+(v(r)-\alpha(r))^2},
\end{align}
the result follows.

\bigskip

\section{Proof of Theorem \ref{equivalence}}\label{see}

We have already proved \eqref{equivalence_1} in Theorem \ref{pressure}, so we just need to prove \eqref{equivalence_2}.
As in the proof of Theorem \ref{pressure}, equation \eqref{equivalence_1} can be read as
\begin{align}
P\pare{\alpha}=-F_{\alpha}\pare{\tilde{u}\pare{\alpha}},
\end{align}
where $\tilde u$ has been defined in the discussion following equation \eqref{zorro}.
Given $u \in {C}\pare{{\T},\pare{-1,1}}$ and taking $\alpha=\tilde{\alpha}\pare{u}$, we get
\begin{equation}
P\pare{\tilde \alpha (u) }=-F_{\tilde \alpha\pare{u}}\pare{u}.
\end{equation}
From the definition of $F_{\tilde \alpha \pare{u}}$, it follows that
\begin{equation}
\int_{\T} f_\beta(u(r))dr=
-P\pare{\tilde \alpha (u) }-\int_{\T} \corch{u(r)-\tilde \alpha (u(r)) }^2 dr.
\end{equation}
It remains to show that 
\begin{equation}\label{misterio}
\int_{\T} f_\beta(u(r))dr \ge -\int_{\T} (u(r)-\alpha(r))^2 dr-P(\alpha)
\end{equation}
for every $\alpha\in{C}\pare{{\T},\pare{-1,1}}$.
Observe that
\begin{align}\label{arcos}
\sum_{\sigma\in\Omega_{\Lambda_\eps,l}\pare{u}}e^{-\beta H_{\Lambda_\eps}^{\nn}\pare{\sigma}}=
Z_{\Lambda_{\eps}, \gamma, \alpha}\sum_{\sigma\in\Omega_{\Lambda_\eps,l}\pare{u}}
\frac{e^{-\beta\corch{H_{\Lambda_\eps}^{\nn}\pare{\sigma}+K_{\Lambda_{\eps},\gamma, \alpha}(\sigma)}}}{Z_{\Lambda_{\eps}, \gamma, \alpha}}
e^{\beta K_{\Lambda_{\eps},\gamma, \alpha}(\sigma)}\leq
Z_{\Lambda_{\eps}, \gamma, \alpha}\sum_{\sigma\in\Omega_{\Lambda_\eps,l}\pare{u}}
e^{\beta K_{\Lambda_{\eps},\gamma, \alpha}(\sigma)},
\end{align}
where $K_{\Lambda_{\eps},\gamma, \alpha}$ is defined in \eqref{K}.
With computations identical to the ones appearing in the proof of Theorem \ref{ttfree}, it follows that
\begin{equation}\label{lim1}
\lim_{l \rightarrow 0} \lim_{\eps \rightarrow 0}-\frac{1}{\beta |\Lambda_\eps|} \log \sum_{\sigma \in \Omega _{l,\eps}(u)} e^{-\beta H^{\nn}_{\Lambda_\eps}(\sigma)} = \int_{\T} f_\beta(u(r))dr
\end{equation}
and
\begin{equation}\label{lim2}
\lim_{l \rightarrow 0} \lim_{\gamma \rightarrow 0} \lim_{\eps \rightarrow 0} \frac{1}{\beta |\Lambda_\eps|}\log \sum_{\sigma \in \Omega_{l,\eps}(u)} e^{\beta K_{\Lambda_{\eps},\gamma, \alpha}(\sigma)} = \int_{\T} (u(r)-\alpha(r))^2 dr .
\end{equation}
Taking 
$-\frac{1}{\beta\abs{\Lambda_\eps}}\log$
in \eqref{arcos}
and passing to the limit using \eqref{lim1} and \eqref{lim2}, 
we obtain \eqref{misterio} and complete the proof of Theorem \ref{equivalence}.

\bigskip

\section{Proof of Theorem \ref{thmcor}}\label{sld}

We prove it first for $\alpha$ and $u$ constant, by taking $\omega$ constant as well.
The general case will follow by applying this case to piecewise constant approximations at an intermediate scale.

\subsection{Constant $u$ and $\alpha$}
We first prove the following exponential bound: for every $\delta>0$, there is a positive number $I(\delta)$ such that
\begin{align}\label{exponential_estimate}
\mu_{\Lambda_{\eps}, \gamma, \alpha}
\pare{\abs{\frac{1}{\abs{\Lambda_\eps}} \sum_{x\in\Lambda_\eps}g\pare{m_{B_{R_\gamma}(x)}\pare{\sigma}}-g\pare{u} } > \delta }\leq
e^{-\abs{\Lambda_\eps}I\pare{\delta}}.
\end{align}
We observe that
\begin{equation}
\Big|\frac{1}{|\Lambda_\eps|} \sum_{x \in \Lambda_\eps} g(m_{B_{R_\gamma}(x)})-g(u)\Big| 
\le \frac{1}{|\Lambda_\eps|} \sum_{x \in \Lambda_\eps} |g(m_{B_{R_\gamma}(x)})-g(u)| \le \frac{K}{|\Lambda_\eps|} \sum_{x \in \Lambda_\eps} |m_{B_{R_\gamma}(x)}-u| \,,
\end{equation}
where we used the fact that $g$ is Lipschitz with constant $K$. 
This implies that
\begin{equation}\label{key}
\mu_{\Lambda_\eps,\gamma,\alpha} \Big( \Big|\frac{1}{|\Lambda_\eps|} \sum_{x \in \Lambda_\eps} g(m_{B_{R_\gamma}(x)})-g(u)\Big| >\delta\Big) \le \mu_{\Lambda_\eps,\gamma,\alpha} \Big( \frac{1}{|\Lambda_\eps|} \sum_{x \in \Lambda_\eps} |m_{B_{R_\gamma}(x)}-u| > \delta/K \Big) \,.
\end{equation}
Now notice that, for every $y \in \mathbb{Z}^d$, we have
\begin{equation}\label{Iandone}
\sum_{z \in \mathbb{Z}^d} J_\gamma(z,y) = \sum_{z \in \mathbb{Z}^d} \gamma^d \phi(\gamma |z-y|) =\int_{\mathbb{R}^d} \phi(r) dr +s(\gamma)=1+s(\gamma) \,,
\end{equation}
where $s(\gamma)\to 0$ when $\gamma\to 0$, uniformly in $y$. 
As a consequence, we have 
\begin{align}\label{main_choice_ofL}
\frac{1}{|B_{R_\gamma}|} \sum_{y \in B_{R_\gamma}(x)} \sigma(y) &=\frac{1}{|B_{R_\gamma}|} \sum_{y \in B_{R_\gamma}(x)} (\sum_{z \in \mathbb{Z}^d} J_\gamma(z,y)-s(\gamma)) \sigma(y) \nonumber \\
&=\frac{1}{|B_{R_\gamma}|}\sum_{z \in B_{R_\gamma}(x)} I_z^\gamma(\sigma) + O(\gamma^{-1}/R_\gamma)  + s(\gamma)O(1)\,,
\end{align}
where we recall the definition of $I_z^\gamma$ in \eqref{Igamma}.
It follows that
\begin{align}
\frac{1}{|\Lambda_\eps|} \sum_{x \in \Lambda_\eps} |m_{B_{R_\gamma}(x)}-u| & \le \frac{1}{|\Lambda_\eps|} \sum_{x \in \Lambda_\eps} |\frac{1}{|B_{R_\gamma}|} \sum_{y \in B_{R_\gamma}(x)}I_y^\gamma(\sigma)-u|+s(\gamma)O(1)+O(\gamma^{-1}/R_\gamma) \nonumber \\
& \le \frac{1}{|\Lambda_\eps|} \sum_{x \in \Lambda_\eps} \frac{1}{|B_{R_\gamma}|} \sum_{y \in B_{R_\gamma}(x)}|I_y^\gamma(\sigma)-u| +s(\gamma)O(1)+O(\gamma^{-1}/R_\gamma) \nonumber \\
&= \frac{1}{|\Lambda_\eps|} \sum_{y \in \Lambda_\eps} |I_y^\gamma(\sigma)-u|+s(\gamma)O(1)+O(\gamma^{-1}/R_\gamma) \,.
\end{align}
The correction term $s(\gamma)O(1)+O(\gamma^{-1}/R_\gamma)$ vanishes when $\gamma \rightarrow 0$. It follows that, for $\gamma$ small enough and for every $\delta>0$, the following estimate holds:
\begin{equation}
\mu_{\Lambda_\eps,\gamma,\alpha} \Big( \frac{1}{|\Lambda_\eps|} \sum_{x\in\Lambda_\eps} |m_{B_{R_\gamma}(x)}-u| > \delta  \Big) \le \mu_{\Lambda_\eps,\gamma,\alpha} \Big(   \frac{1}{|\Lambda_\eps|} \sum_{x\in\Lambda_\eps} |I_x^\gamma(\sigma)-u| > \delta/2  \Big)\,.
\end{equation}
To estimate the latter expression, we observe that, $\forall \delta'>0$,
\begin{align}
\frac{1}{|\Lambda_\eps|} \sum_{x\in\Lambda_\eps} |I_x^\gamma(\sigma)-u| & \le 2 \frac{ |\{x: |I_x^\gamma(\sigma)-u|>\delta' \}|}{|\Lambda_\eps|}+\delta' \,.
\end{align}
If we choose $\delta'=\delta/2$, we obtain
\begin{equation}
\mu_{\Lambda_\eps,\gamma,\alpha} \Big(\frac{1}{|\Lambda_\eps|}\sum_{x\in\Lambda_\eps} |I_x^\gamma(\sigma)-u| > \delta \Big) \le \mu_{\Lambda_\eps,\gamma,\alpha} \Big( \frac{ |\{x: |I_x^\gamma(\sigma)-u|>\delta/2 \}|}{|\Lambda_\eps|} > \delta/4  \Big) \,.
\end{equation}
Thus, we reduced the problem to the following lemma, whose proof is given in the Appendix \ref{lem51}:

\begin{lemma}\label{pedro}
For every $c, \delta>0$ and $\gamma$ small enough, we have that
\begin{align}
\mu_{\Lambda_{\eps}, \gamma, \alpha}
\pare{
{\abs{\llav{x\in\Lambda_\eps: \abs{I_x^{\gamma}\pare{\sigma}-u}>\delta}}} >c{\abs{\Lambda_\eps}}}
\leq e^{-\abs{\Lambda_\eps}\beta c\delta^2/2},
\end{align}
where $I_x^{\gamma}$ is defined in \eqref{Igamma}.
\end{lemma}

\medskip

\subsection{The inhomogeneous case.}\label{sec5.2}
Like before, we consider a macroscopic scale characterized by the parameter $l$, which we take to be equal to $2^{-p}$ for $p\in\mathbb N$.
Recall that ${\mathscr C}_l$ is the corresponding macroscopic partition of $\T$ into 
$N_l \defi l^{-d}$ many sets denoted by $C_{l,k}$,
$k=1,\ldots, N_l$. We denote their microscopic versions by $\Delta_{\eps^{-1}l,k}$.
Let $u^{(l)}$ and $\omega^{(l)}$ respectively be the piece-wise constant approximations of $u$ and $\omega$ defined as \eqref{ul}. 
Since $g$ is bounded and continuous and $\omega$ is uniformly continuous, we have
\begin{eqnarray}\label{p1}
&& \abs{
\frac{1}{\abs{\Lambda_\eps}}
\sum_{x\in\Lambda_\eps}\omega\pare{\eps x}g(m_{B_{R_\gamma}(x)} (\sigma))
-\int_{\T} \omega(r) g(u(r))dr }   \nonumber\\ 
& = &
\abs{\frac{1}{\abs{\Lambda_\eps}} \sum_{x \in \Lambda_\eps} \omega^{(l)}(\eps x) g(m_{B_{R_\gamma}(x)}(\sigma))-\int_{\T} \omega^{(l)}(r) g(u^{(l)}(r))dr} +O(l)
\end{eqnarray}

For $x\in \Delta_{\eps^{-1}l,i}$, let
$\tilde m_{B_{R_\gamma}(x)}$
be the magnetization considering periodic boundary conditions in $\Delta_{\eps^{-1}l,i}$.
Note that $\tilde{m}_{B_{R_\gamma}(x)}$ coincides with ${m}_{B_{R_\gamma}(x)}$ if the distance between $x$ and 
$\Lambda_\eps\setminus \Delta_{\eps^{-1}l,i}$ is larger than $R_\gamma$.
Then, since $g$ is Lipschitz,
we have
\begin{align}
\sum_{x\in \Delta_{\eps^{-1}l,i}}g\pare{m_{B_{R_\gamma}(x)}}=
\sum_{x\in \Delta_{\eps^{-1}l,i}}g\pare{\tilde m_{B_{R_\gamma}(x)}}+O\pare{(\eps^{-1}l)^{d-1}L}.
\end{align} 
Replacing in \eqref{p1} and splitting over the boxes of the partition ${\mathscr D}_{\eps^{-1}l}$, we obtain
\begin{equation}\label{p2}
\eqref{p1}
\leq 
\frac{1}{N_l}\sum_{i=1}^{N_l} |\bar\omega^{(l)}_i| \abs{\frac{1}{|\Delta_{\eps^{-1}l,i}|}\sum_{x \in \Delta_{l,i}}  g\pare{\tilde m_{B_{R_\gamma}(x)}}- g(\bar u^{(l)}_i)}+O(l)
+O(\frac{L}{\eps^{-1}l}).
\end{equation}
Then, defining
\begin{equation}
Y_i \defi 
\abs{\frac{1}{|\Delta_{\eps^{-1}l,i}|}\sum_{x \in \Delta_{\eps^{-1}l,i}}  g\pare{\tilde m_{B_{R_\gamma}(x)}}- g(\bar u^{(l)}_i)},
\end{equation}
for $l$ and $\eps$ small enough, we have
\begin{align}
&\mu_{\Lambda_{\eps}, \gamma, \alpha} \pare{ 
\abs{\frac{1}{\abs{\Lambda_\eps}}
\sum_{x\in\Lambda_\eps}\omega\pare{\eps x}g\left(m_{B_{R_\gamma}(x)}\pare{\sigma}\right)
-\int_{\T} \omega(r) g\pare{u(r)} dr}>\delta }
\leq 
\mu_{\Lambda_{\eps}, \gamma, \alpha} \pare{ \frac{1}{N_l} \sum_{i=1}^{N_l} Y_i > \delta/2 }.
\end{align}
We notice that, for $\zeta >0$,
\begin{align}
\frac{1}{N_l} \sum_{i=1}^{N_l} Y_i &=\frac{1}{N_l} \sum_{i:\, Y_i>\zeta} Y_i +\frac{1}{N_l}\sum_{i:\, Y_i \le \zeta} Y_i 
 \le 2 \norm{g}_\infty \frac{\abs{ \llav{ i: Y_i >\zeta }} }{N_l}+\zeta .
\end{align}
Choosing $\zeta<\delta/2$ and
setting $\delta' \defi \frac{1}{2 \norm{g}_\infty} \pare{ \delta/2-\zeta }$, we have
\begin{equation} \label{ale001}
\mu_{\Lambda_{\eps}, \gamma, \alpha} \pare{ \frac{1}{N_l} \sum_{i=1}^{N_l} Y_i > \delta/2 } \le \mu_{\Lambda_{\eps}, \gamma, \alpha} \pare{\abs{ \llav{i: \, Y_i >\zeta }} \geq \delta' N_l  }.
\end{equation}
To proceed, we
apply the result obtained in the first step for $\alpha$ and $u$ constant.
For this purpose, we 
define a new probability measure $\tilde \mu_{\Lambda_{\eps}, \gamma, \alpha^{(l)}}$
defined on the union of the boxes $\Delta_{\eps^{-1}l,i}$ with periodic boundary conditions in each of
them and with external field $\alpha^{(l)}$ as defined in \eqref{bara}.
Then, by neglecting the interactions between the boxes $\Delta_{\eps^{-1}l,i}$, $i=1,\ldots,N_l$, 
for any set $B$ we obtain that
\begin{align} \label{mu_le_mu_tilde}
\mu_{\Lambda_{\eps}, \gamma, \alpha}(B) \le e^{|\Lambda_\eps| O(l)+O(\eps^{1-d} l^{-1}) +O(\eps^{1-d} l^{-1} \gamma^{-1})} \tilde \mu_{\Lambda_{\eps}, \gamma, \alpha^{(l)}} (B).
\end{align}
Let us denote by $\lceil \delta' N_l \rceil$ the smallest integer not smaller than $\delta' N_l$. It follows from \eqref{exponential_estimate} that there exists $I(\zeta)>0$ such that
\begin{equation} \label{ale002}
 \tilde \mu_{\Lambda_{\eps}, \gamma, \alpha^{(l)}} \pare{ |\{ i: Y_i >\zeta \} | \ge \delta' N_l  } 
 \le \binom{N_l}{\lceil \delta' N_l \rceil} 
 \prod_{i=1}^{\lceil \delta' N_l \rceil}
 \tilde\mu_{\Delta_{\eps^{-1}l,i}, \gamma, \bar\alpha^{(l)}_i} (Y_i>\zeta)
 \le \binom{N_l}{\lceil \delta' N_l \rceil} e^{-\eps^{-d}  \delta'  I(\zeta)} \,.
\end{equation}
Noting that
\begin{equation}
\binom{N_l}{\lceil \delta' N_l \rceil}=e^{N_l(-\delta'\log\delta'-(1-\delta')\log(1-\delta'))+c\log N_l},
\end{equation}
from \eqref{ale001}, \eqref{mu_le_mu_tilde} and \eqref{ale002}, we obtain the following estimate
\begin{equation}
\mu_{\Lambda_{\eps}, \gamma, \alpha} \pare{ \frac{1}{N_l} \sum_{i=1}^{N_l} Y_i > \delta/2 } \le \exp \llav{ \eps^{-d} O(l)+O(\eps^{1-d} l^{-1}) +O(\eps^{1-d} l^{-1} \gamma^{-1}) + cN_l -\eps^{-d} \delta' I(\zeta) } \,.
\end{equation}
If we choose $l$ small enough, the coefficient of $\eps^{-d}$ inside the exponential is negative and thus we obtain 
\begin{equation}
\lim_{\eps \rightarrow 0} \mu_{\Lambda_{\eps}, \gamma, \alpha}\Big( \frac{1}{N_l} \sum_{i=1}^{N_l} Y_i > \delta/2 \Big)=0 ,
\end{equation}
concluding the proof of the inhomogeneous case as well as of Theorem \ref{thmcor}.\qed

\bigskip

\section{Young-Gibbs measures, proof of Theorem \ref{mainthm}}\label{sYG}

As mentioned before, the first case is just a restatement of Theorem \ref{thmcor}.
For the second case,
it suffices to prove an exponential bound for the constant case and then the inhomogeneous case follows by
the strategy in Subsection \ref{sec5.2}.
The last case is a direct consequence and it will be given at the end of this section.
Hence, for the rest of this section, we restrict ourselves to constant $\alpha$, $u$ and $\omega$.
We first prove the case $|u|>m_\beta$ and then the more difficult one: $|u|\leq m_\beta$.
The hypotheses over the dimension and $\beta$ are needed only in the second case.

\textbf{Case $|u|>m_\beta$.} 
Let $f$ be a local function and $f_x$  its translation by $x\in\Lambda_\eps$.
For simplicity of notation, we use $f$ instead of $g(m_{B_R})$.
Then, for $\omega$ constant and for fixed $\delta>0$, 
it suffices to prove an exponential bound for $\mu_{\Lambda_\eps,\gamma,\alpha}(E_{\delta})$,
where
\begin{align}\label{61}
E_{\delta} \defi  \Big(\Big|\frac{1}{|\Lambda_\eps|}\sum_{x\in\Lambda_\eps}f_x-\E_{\mu^{\nn}_h}(f) \Big|>\delta\Big),
\end{align}
with $h$ such that $\E_{\mu^{\nn}_h}(\sigma_{0})=u$, i.e., for $h \defi f'_{\beta}(u)$.
We expand the Hamiltonian $H_{\Lambda_\eps,\gamma,\alpha}$ as follows:
\begin{align}\label{trick}
H_{\Lambda_\eps,\gamma,\alpha} = H_{\Lambda_\eps}^{\nn}
+\sum_{x\in\Lambda_\eps}[I^{\gamma}_x-u]^2
+2(u-\alpha)\sum_{x\in\Lambda_\eps}I^{\gamma}_x-2(u-\alpha)u|\Lambda_\eps|+(\alpha-u)^2|\Lambda_\eps|.
\end{align}
When considering the corresponding measure, the constant terms cancel with the normalization. 
Note that 
\begin{equation}\label{cor_ext}
\alpha \defi \tilde\alpha(u)=u+\frac 12 f'_{\beta}(u)\Rightarrow
2(\alpha-u)=h.
\end{equation}
Hence, recalling \eqref{guante} with the above $h$, we consider the following Hamiltonian:
\begin{align}
\hat H_{\Lambda_\eps,\gamma,\alpha} \defi 
H_{\Lambda_\eps,h}^{\nn}
+\sum_{x\in\Lambda_\eps}[I^{\gamma}_x-u]^2.
\end{align}
To treat the second term, for some parameter $\zeta>0$, we consider the random variable
\begin{align}
D_\zeta(\sigma) \defi  \frac{1}{|\Lambda_\eps|}|\{ x\in\Lambda_\eps:|I^{\gamma}_x(\sigma)-u| >\zeta  \}|,
\end{align}
which gives the density of bad Kac averages.
Then, using the inequality
\begin{align}\label{spl}
\mu_{\Lambda_\eps,\gamma,\alpha}(E_{\delta})\leq 
\mu_{\Lambda_\eps,\gamma,\alpha}
(E_{\delta},D_\zeta\leq \delta')
+\mu_{\Lambda_\eps,\gamma,\alpha}
(D_\zeta> \delta')
\end{align}
and Lemma \ref{pedro} (for appropriate choice of $\zeta$ and $\delta'$), the problem
reduces to finding exponential bounds for the first term.
Notice that, using \eqref{Iandone} and \eqref{main_choice_ofL}, we have
\begin{align}
\sum_{y\in\Lambda_\eps}I_y^{\gamma}\pare{\sigma}=
\sum_{y\in\Lambda_\eps}\sum_{x\in\Lambda_\eps}J_{\gamma}\pare{x,y}\sigma(x)=
\sum_{x\in\Lambda_\eps}\sigma(x)\sum_{y\in\Lambda_\eps}
J_{\gamma}\pare{x,y}=
\sum_{x\in\Lambda_\eps}\sigma(x)+s\pare{\gamma}O\pare{1}\abs{\Lambda_\eps}
\end{align}
for some $s(\gamma)\to 0$ as $\gamma\to 0$.
Then the first term on the right-hand side of \eqref{spl} is bounded by
\begin{align}\label{poiu}
\frac{e^{ C_1 s(\gamma)|\Lambda_\eps|}}{\hat Z_{\Lambda_\eps,\gamma,\alpha}} \,
\sum_{\sigma\in\Omega_{\Lambda_\eps}}e^{-\beta \hat H_{\Lambda_\eps,\gamma,\alpha}(\sigma)}\1_{ E_{\delta} } \1_{ \{ D_\zeta(\sigma)\leq \delta' \}},
\end{align}
where $C_{1}$ is a positive constant and
where $\hat Z_{\Lambda_\eps,\gamma,\alpha}$ is the partition function associated to $\hat H_{\Lambda_\eps,\gamma,\alpha}$.
For $\sigma$ such that $D_\zeta(\sigma)\leq \delta'$, we have
\begin{align}
H_{\Lambda_\eps,h}^{\nn}(\sigma)\leq
\hat H_{\Lambda_\eps,\gamma,\alpha}(\sigma)\leq H_{\Lambda_\eps,h}^{\nn}(\sigma)+|\Lambda_\eps|( C_2 \delta'+\zeta^2),
\end{align}
so \eqref{poiu} is bounded by
\begin{align}\label{ffgg}
e^{C_3(s(\gamma)+\delta'+\zeta^2)|\Lambda_\eps|} \
\frac{\displaystyle{
\sum_{\sigma\in\Omega_{\Lambda_\eps}}e^{-\beta  H_{\Lambda_\eps,h}^{\nn}(\sigma)}\1_{E_{\delta} } \1_{ \{ D_\zeta(\sigma)\leq \delta' \}}}}
{\displaystyle{\sum_{\sigma\in\Omega_{\Lambda_\eps}}e^{-\beta  H_{\Lambda_\eps,h}^{\nn}(\sigma)}\1_{\{ D_\zeta(\sigma)\leq \delta' \}}}}
=e^{ C_3(s(\gamma)+\delta'+\zeta^2)|\Lambda_\eps|} \mu^{\nn}_{\Lambda_\eps,h}(E_{\delta}|D_\zeta\leq \delta').
\end{align}
From Lemma \ref{pedro},  we have that $\mu^{\nn}_{\Lambda_\eps,h}(D_\zeta\leq \delta')>\frac{1}{2}$, for $\eps$ small enough.
Moreover, it is a standard result 
that there exists $C_4(\delta)>0$ such that
\begin{align}\label{tyty}
\mu^{\nn}_{\Lambda_\eps,h}(E_{\delta})\leq e^{-C_4(\delta)|\Lambda_\eps|},
\end{align}
for $\eps$ small enough.
For the exponential bound \eqref{tyty}, we refer to \cite{E}, Theorem V.6.1.
Actually, this theorem gives the result for $f$ a local magnetization, that is, for $f$ of the form $f(\sigma)=\frac{1}{|\Delta|}\sum_{x\in\Delta}\sigma(x)$; in our case, this is enough as every local function can be written as a linear combination of local magnetizations.
Under these considerations, by appropriately choosing $\zeta$, $\delta'$ and 
for $\gamma$ small, the right hand side of \eqref{ffgg} is bounded by $3e^{-C_5(\delta)|\Lambda_\eps|}$ for some constant $C_{5}(\delta)>0$,
and the result follows.

\medskip

\textbf{Case $|u|\leq m_\beta$.} 
In this case, for $f$ a local function, we seek an exponential bound for
\begin{align}\label{kjkj}
\mu_{\Lambda_\eps,\gamma,\alpha}\Big( \Big|\frac{1}{|\Lambda_\eps|}\sum_{x\in\Lambda_\eps}f_x(\sigma)-\E_{G_u}(f_x)\Big|>\delta\Big),
\end{align}
where $G_u \defi  \lambda_u\mu^{\nn}_++(1-\lambda_u)\mu^{\nn}_-$ with $\lambda_u$ as in $\eqref{p1p1}$. Comparing to \eqref{61}, we notice that, instead of the measure $\mu^{\nn}_{h}$
with the external field corresponding to $u$, we have the canonical 
measure $G_{u}$.
Hence, in order to work with realizations of the measure $G_{u}$, we need to introduce a 
scale $K$ and prove that, for boxes in this scale, the relevant measures are
$\mu^{\nn}_+$ or $\mu^{\nn}_-$ and that they appear with a percentage that agrees with the
overall fixed magnetization $u$.
There are two main obstacles: the first is that the Kac term in the original measure
cannot directly fix the magnetization via large deviations as in Theorem \ref{thmcor},
since  we are looking at averages in a smaller scale than $\gamma^{-1}$; in particular, \eqref{main_choice_ofL} is not true.
The second is to show that, in the smaller scale $K$, only the nearest-neighbour part of the Hamiltonian is effective.
Hence, we introduce another scale $L\gg K$,
in which the Kac term acts to all spins in the same way.
Then inside the box only the nearest-neighbour interactions are relevant.

To proceed with this strategy,
we fix a microscopic scale $K$ of the form $2^m$ and call $\Delta_{K,1},\ldots,\Delta_{K,N_K}$ the partition of $\Lambda_\eps$ into 
\begin{equation}\label{N_K}
N_K \defi (\eps K)^{-d}
\end{equation}
boxes of side-length $K$.
We call $\Delta_{K,i}^0$ the boxes with the same center as $\Delta_{K,i}$ and distance $\sqrt{K}$
from their complement $\Delta^c_{K,i}$.
We next introduce the notions of ``circuit'' and of ``bad box''.

\begin{definition}[circuit]\label{circuit}
It is easier to define the lack of circuit.
For a sign $\tau=\pm$, we say that a configuration $\sigma\in\{-1,1\}^{\Lambda_\eps}$ does not have a $\tau$-circuit in $\Delta_{K,i}$ if there exists a path of vertices $\{ x_1,\ldots,x_k \}\subset \Delta_{K,i}\setminus \Delta_{K,i}^0$ such that $d(x_1,\Delta_{K,i}^c)=1$, $d(x_k,\Delta_{K,i}^0)=1$, $d(x_i,x_{i+1})=1$ for every $i=1,\ldots,k-1$, and $\sigma_{x_i}=-\tau$ for every $i=1,\ldots,k$.
In other words, if the connected components of $-\tau$ that intersects the boundary of $\Delta_{K,i}$ do not intersect $\Delta_{K,i}^0$.
If we are not interested in distinguishing the sign of the circuit, we just say that $\sigma$ has a circuit.
\end{definition}

Observe that the existence of a $\tau$-circuit can be decided from the outside configuration.

\begin{definition}\label{defbad}
Given some precision $\zeta>0$, a box $\Delta_{K,i}$ is called $\zeta$-bad for a configuration $\sigma$ if
\begin{itemize}
\item $\sigma$ does not have a circuit in $\Delta_{K,i}$, or if
\item $\sigma$ has a circuit in $\Delta_{K,i}$ but
\begin{equation}\label{criterion}
\min_{\tau= \pm}
\Big| \frac{1}{|\Delta_{K}|} \sum_{x\in \Delta_{K,i}}
f_x- 
\E_{\mu^{\nn}_{0,\tau}}(f) \Big| >\zeta.
\end{equation}
\end{itemize}
On the other hand, we call a box $\zeta$-good if it is not $\zeta$-bad.
We can further specify it saying it is $(\zeta,\tau)$-good if
\begin{align}\label{defgood}
\Big| \frac{1}{|\Delta_{K}|} \sum_{x\in \Delta_{K,i}}
f_x- 
\E_{\mu^{\nn}_{0,\tau}}(f) \Big| \leq\zeta.
\end{align}
\end{definition}

Let $N^{\bad}_{K,\zeta}$ and $N_{K,\zeta,\tau}^\good$ be the number of $\zeta$-bad and $(\zeta,\tau)$-good boxes, respectively.
To conclude the proof of the case $|u|\leq m_{\beta}$, it suffices to prove that
the probability of having a large density of $\zeta$-bad boxes is small and that the density of $(\zeta,+)$-good boxes is $\lambda_u$;
this is the content of the following lemma.

\begin{lemma}\label{prop_lem_1} 
For every $\zeta,\delta>0$, there exists $C(\zeta,\delta)>0$ such that the following exponential bounds  hold for every $\eps$ and $\gamma$ small enough and $K$ large enough:
\begin{align}
&\textnormal{(i)} \ \ \mu_{\Lambda_\eps,  \gamma,\alpha}
\bigg(\frac{N^{\bad}_{K,\zeta}}{N_K}>\delta \bigg)\le e^{-C(\zeta,\delta)|\Lambda_\eps|};\label{prove}
\\[0.2cm]
& \textnormal{(ii)} \ \ \mu_{\Lambda_\eps, \gamma,\alpha}
\bigg(
\bigg|\frac{N_{K,\zeta,+}^\good}{N_K}-\lambda_u\bigg|>\delta
\bigg)\le e^{-C(\zeta,\delta)|\Lambda_\eps|}.\label{prove1}
\end{align}
\end{lemma}

\bigskip
Before giving its proof, we see how the case $|u|\leq m_{\beta}$ follows from it.
For $\zeta,\delta'>0$ (they will later depend on $\delta$), \eqref{kjkj} 
is bounded by
\begin{align}\label{urur}
\mu_{\Lambda_\eps,\gamma, \alpha} \bigg( \bigg| \frac{1}{|\Lambda_\eps|}\sum_{x\in\Lambda_\eps}f_x-\E_{G_u}(f) \bigg|>\delta,\frac{N^\bad_{K,\zeta}}{N_K}\leq \delta'\bigg)+
\mu_{\Lambda_\eps,\gamma, \alpha}\bigg(\frac{N^\bad_{K,\zeta}}{N_K}> \delta'\bigg).
\end{align}
The exponential bound for the second term is given by Lemma \ref{prop_lem_1}.
To control the first one, we decompose the average as follows:
\begin{align}
\frac{1}{|\Lambda_\eps|}\sum_{x\in\Lambda_\eps}f_x-\E_{G_u}(f) 
=\frac{1}{N_K}\sum_{i=1}^{N_K}\bigg(\frac{1}{|\Delta_K|}\sum_{x\in\Delta_{K,i}}f_x-\E_{G_u}(f)\bigg).
\end{align}
Take $\delta'=\frac{\delta}{4\norm{f}_\infty}$ and observe that, for $\sigma$  such that $\frac{N^\bad_{K,\zeta}(\sigma)}{N_K}\leq \delta'$, we have
\begin{align}
\Big|\frac{1}{|\Lambda_\eps|}\sum_{x\in\Lambda_\eps}f_x-\E_{G_u}(f)\Big|\leq
\Big|\frac{1}{N_K}\sum_{i: \, \Delta_{K,i}\text{ is }\zeta\text{-good}}
\Big(\frac{1}{|\Delta_K|}\sum_{x\in\Delta_{K,i}}f_x-\E_{G_u}(f)\Big)\Big|+\frac{\delta}{2}.
\end{align}
Then the first term of \eqref{urur} is bounded by
\begin{align}\label{kmkm}
\mu_{\Lambda_\eps,\gamma, \alpha} \bigg(
\bigg|\frac{1}{N_K}\sum_{i: \, \Delta_{K,i}\text{ is }\zeta\text{-good}}
\bigg(\frac{1}{|\Delta_K|}\sum_{x\in\Delta_{K,i}}f_x-\E_{G_u}(f)\bigg)\bigg|
>\frac{\delta}{2},\frac{N^\bad_{K,\zeta}}{N_K}\leq \delta'\bigg).
\end{align}
Subtracting and adding $\E_{\mu^{\nn}_\tau}(f)$, we have 
\begin{align}
&\bigg|\frac{1}{N_K} \  \sum_{i: \, \Delta_{K,i}\text{ is }\zeta\text{-good}} \ 
\bigg(\frac{1}{|\Delta_K|}\sum_{x\in\Delta_{K,i}}f_x-\E_{G_u}(f)\bigg)\bigg|\nonumber
\\[0.2cm]
& \ \ \ \leq \ \frac{1}{N_K} \sum_{\tau=\pm}   \ \sum_{i: \, \Delta_{K,i}\text{ is }(\zeta,\tau)\text{-good}}
\ \bigg|\frac{1}{|\Delta_K|}\sum_{x\in\Delta_{K,i}}f_x-\E_{\mu^{\nn}_\tau}(f)\bigg|
+2\norm{f}_\infty \bigg|\frac{N_{K,\zeta,+}^\good}{N_K}-\lambda_u\bigg|.
\end{align}
Choosing $\zeta=\frac{\delta}{4}$, the first term in the last expression is smaller than $\frac{\delta}{4}$, thus \eqref{kmkm} is bounded by
\begin{align}
\mu_{\Lambda_\eps,\gamma,\alpha}\Big(
2\norm{f}_\infty \Big|\frac{N_{K,\zeta,+}^\good}{N_K}-\lambda_u\Big|>\frac{\delta}{4}
,\frac{N^\bad_{K,\zeta}}{N_K}\leq \delta'\Big)\leq 
\mu_{\Lambda_\eps,\gamma,\alpha}
\bigg(2\norm{f}_\infty \Big|\frac{N_{K,\zeta,+}^\good}{N_K}-\lambda_u\Big|>\frac{\delta}{4}\bigg)
\end{align}
which, by Lemma \ref{prop_lem_1}, decays exponentially.
\qed

\begin{proof}[Proof of Lemma \ref{prop_lem_1}] \

(i) \ \ We first notice that the criterion for a box to be ``bad'' is based only on the nearest-neighbour interaction part of the measure.
Therefore, instead of estimating \eqref{prove} using $\mu_{\Lambda_\eps, \gamma,\alpha}$,
we reduce ourselves to an estimate using only the Ising part.
To do that, we introduce another intermediate scale $L$ of order $\gamma^{-1+a}$, for $a>0$, 
and we first condition over all possible values of the
magnetization in this scale: we divide $\Lambda_\eps$ into boxes $\Delta_{L,1},\ldots,\Delta_{L,N_L}$, $N_L=(\eps L)^{-d}$ (recall \eqref{N_K}) and, in each box $\Delta_{L,i}$,
the new order parameter 
$m_{\Delta_{L,i}}(\sigma)$ takes values in
$I_{|\Delta_{L,i}|}$.
We denote this new configuration space by $\mathcal M_{L} \defi \prod_{i=1}^{N_L}I_{|\Delta_{L,i}|}$.
Then, by conditioning on a set of configurations with a given
average magnetization in $\mathcal M_L$,
the Kac part of the Hamiltonian is essentially constant so  we are only left 
with the nearest-neighbour
interaction.

To proceed with this plan, we follow the coarse-graining procedure as in Section \ref{ttfree};
 recall the effective interaction $ \bar J_{\gamma}^{(L)}$ in the new scale $L$ given in \eqref{Jbar}.
For $\eta \defi \{\eta_i\}_i\in\mathcal M_{L}$, recalling \eqref{cg},
we denote the new coarse-grained Hamiltonian
\begin{equation}\label{cgH}
\bar K^{(L)}_{\Lambda_{\eps},\gamma,\alpha}(\eta) \defi L^d \sum_i\bigg(\sum_{j}  \bar J_{\gamma}^{(L)}(i,j)\eta_j-\alpha\bigg)^2
\end{equation}
(note that $\alpha$ is constant).
Recalling the error \eqref{Kacerror1},
for $L=\gamma^{-1+a}$, we obtain that
\begin{align}\label{reduction}
\mu_{\Lambda_\eps,\gamma,\alpha}(N^{\bad}_{K,\zeta}>\delta N_K)
 &= 
\sum_\sigma \mathbf 1_{\{N^{\bad}_{K,\zeta}>\delta N_K\}}
\frac{1}{Z_{\Lambda_\eps,\gamma,\alpha}}e^{-\beta H_{\Lambda_{\eps}}^{\nn}(\sigma)}
e^{-\beta K_{\Lambda_{\eps},\gamma, \alpha}(\sigma)}\nonumber\\
&=\sum_{\eta\in\mathcal M_{L}}e^{-\beta \bar K^{(L)}_{\Lambda_{\eps},\gamma, \alpha}(\eta)}
\frac{1}{Z_{\Lambda_\eps,\gamma,\alpha}}
\sum_{\substack{\sigma: \,\forall i,\\ m_{\Delta_{L,i}}(\sigma_i)=\eta_i}}
\mathbf 1_{\{N^{\bad}_{K,\zeta}>\delta N_K\}}
e^{-\beta H_{\Lambda_{\eps}}^{\nn}(\sigma)}
e^{C\gamma^{a}|\Lambda_\eps|},
\end{align}
where
\begin{eqnarray}\label{splitpartition}
Z_{\Lambda_\eps,\gamma,\alpha}
& = &
\sum_{\eta\in\mathcal M_{L}}e^{-\beta \bar K^{(L)}_{\Lambda_{\eps},\gamma, \alpha}(\eta)}
\sum_{\substack{\sigma: \,\forall i,\\ m_{\Delta_{L,i}}(\sigma_i)=\eta_i}}
e^{-\beta H_{\Lambda_{\eps}}^{\nn}(\sigma)}
e^{C\gamma^{a}|\Lambda_\eps|}\nonumber\\
& = &
e^{C \gamma^{a} |\Lambda_\eps|}
Z^{\nn}_{\Lambda_\eps,0}
\sum_{\eta\in\mathcal M_{L}}e^{-\beta \bar K^{(L)}_{\Lambda_{\eps},\gamma, \alpha}(\eta)}
\mu^{\nn}_{\Lambda_\eps,0}(\{m_{\Delta_{L,i}}=\eta_{i}\}_{i=1}^{N_{L}}).
\end{eqnarray}
Note that in the splitting in \eqref{reduction} we do not specify the boundary conditions, as with an extra
lower order (surface) error we can choose them ad libitum.
Hence, we have to estimate $\mu_{\Lambda_\eps,0}^{\nn}(\{N^{\bad}_{K,\zeta}>\delta N_K\})$.
We split it into a product over the measures $\mu_{\Delta_{L,i},0,+}^{\nn}$ assuming 
$+$ boundary conditions and making an error of lower order. 
Then, we focus in a box $\Delta_L$ and 
denote by $N_{K,L}$ (respectively $N^{\bad}_{K,L,\zeta}$) the number of boxes 
(respectively bad boxes) 
of size $K$ in $\Delta_L$.
In order to conclude,
it suffices to show that there is $r(\delta, \zeta, K)>0$ such that
\begin{equation}\label{step1}
\mu_{\Delta_L,0,+}^{\nn}(\{N^{\bad}_{K,L,\zeta}>\delta N_{K,L}\})\leq e^{-N_{K,L} r(\delta,\zeta,K)}.
\end{equation}
The proof of \eqref{step1} is lengthy and it is outlined below, after the end of the proof of Lemma \ref{prop_lem_1}.
Furthermore,
this decaying estimate should win against the accumulating errors of the order $\gamma^a |\Lambda_\eps|$ in \eqref{reduction} and \eqref{splitpartition}.
This is true since 
$\gamma^a K^d
\ll  r(\delta,\zeta,K)$, for $\gamma$ small enough, after using the fact that $|\Delta_L|=N_{K,L} K^d$.
We also need a lower bound of \eqref{splitpartition}. For that, it suffices to show that for every $i$:
\begin{equation}\label{lb}
\mu^{\nn}_{\Delta_{L,i},0}(\{m_{\Delta_{L,i}}=\eta_{i}\})
\geq e^{-c L^{d-1}}.
\end{equation}
The proof is given in Appendix \ref{denominator}, concluding the proof of item (i) of Lemma \ref{prop_lem_1}.
\bigskip

(ii) \ To prove \eqref{prove1}, for $u$ constant, in a box $\Lambda_\eps$ we have: 
\begin{eqnarray}\label{split}
\frac{1}{|\Lambda_\eps|}
\sum_{x\in \Lambda_\eps}\sigma(x)-u
&=&
\frac{1}{N_K}\sum_{i: \, \Delta_{K,i}\text{ is }(\zeta,+)\text{-good}}
\Big(
\frac{1}{|\Delta_K|}\sum_{x\in \Delta_K}\sigma(x)-m_\beta\Big)
+\Big(\frac{N_{K,\zeta,+}^\good}{N_K}-\lambda_u
\Big)
m_\beta
\nonumber\\
&& +
\frac{1}{N_K}
\sum_{i: \, \Delta_{K,i}\text{ is }(\zeta,-)\text{-good}}
\Big(
\frac{1}{|\Delta_K|}\sum_{x\in \Delta_K}\sigma(x)+m_\beta\Big)
+\Big(\frac{N_{K,\zeta,-}^\good}{N_K}-(1-\lambda_u)
\Big)
(-m_\beta)
\nonumber\\
&&
+
\frac{1}{N_K}
\sum_{i: \, \Delta_{K,i}\text{ is }\zeta \text{-bad}}
\frac{1}{|\Delta_K|}\sum_{x\in \Delta_K}\sigma(x),
\end{eqnarray}
Moreover,
\begin{equation*}
\frac{N_{K,\zeta,-}^\good}{N_K}-(1-\lambda_u)=-\left(\frac{N_{K,\zeta,+}^\good}{N_K}-\lambda_u\right)-\frac{N^{\bad}_{K,\zeta}}{N}.
\end{equation*}
From Definition \ref{defbad}, in the good boxes we have a circuit of $\pm$ spins.
Then, using \eqref{prel1}, for every $x\in \Delta^0_{K}$ we have that
\begin{equation}\label{est_sign}
\E_{\mu^{\nn}_{\Delta_K,0,\pm}}[\sigma(x)] = \pm m_\beta + K^{d}e^{-CK},
\end{equation}
for $K$ large and a generic box $\Delta_K$.
We consider the measure $\mu_{\Lambda_\eps,\gamma,\alpha}$ 
and use the estimate \eqref{exponential_estimate}.
We split the measure over the boxes $(\Delta_{K,i})_i$ like previously and, using \eqref{split}
as well as the estimate \eqref{prove}, we obtain \eqref{prove1}.
This concludes the proof of Lemma \ref{prop_lem_1}.
\end{proof}

\medskip

In the sequel, we first prove the remaining estimate \eqref{step1}.
Here we present the strategy and state the main lemmas. For the proofs we refer to Appendix \ref{pfoflem1} and \ref{pfoflem2}.
The section will conclude with the proof of \eqref{third}.

\medskip

\noindent {\it Proof of \eqref{step1}.}
Given a box $\Delta_L$, let $I\subset \{1,\ldots, N_{K,L}\}$ denote the indices of the boxes $\Delta_{K}$ within it.

\begin{definition}\label{twosets}
Given $I\subset \{1,\ldots, N_{K,L}\}$ and $\underline a\in\{-,+\}^{I}$, we define $\mathcal X'_{I,\underline a}$ to be the set of configurations
where there is some circuit around $\Delta_{K,i}^0$
for all $i\in I$ and \eqref{criterion} is true.
On the other hand, we define $\mathcal X''_I$ to be the set of configurations for
which there is no circuit for any of the boxes in $I$.
\end{definition}
Asking for more than $\delta N_{K,L}$, $0<\delta<1$, many bad boxes is equivalent to the fact that
at least one of the two cases described in Definition \ref{twosets}
has to occur more than $\frac{\delta N_{K,L}}{2}$, hence:
\begin{equation}\label{split1_5}
\{N^{\bad}_{K,L,\zeta}>\delta N_{K,L}\}
\subset
\Big(
\bigcup_{(I,\underline a): |I|\geq\frac{\delta}{2} N_{K,L}}\mathcal X'_{I,\underline a}
\Big)
\cup
\Big(
\bigcup_{I: |I|\geq\frac{\delta}{2} N_{K,L}}\mathcal X''_{I}
\Big)
\end{equation}
To estimate the first contribution, we have the following lemma:

\begin{lemma}\label{lem1}
Consider a box $\Delta_L$ divided into $N_{K,L}$ smaller boxes $\Delta_K$, with $K\ll L$.
There is a positive constant $c$ so that, for any $I\subset\{1,\ldots,N_{K,L}\}$ and $i\in I$, the following is true:
\begin{equation}\label{chi'}
\mu_{\Delta_L,0,+}^{\nn}(\mathcal X'_{I,\underline{a}})\leq c \zeta^{-2}K^{-d}
\mu_{\Delta_L,0,+}^{\nn}(\mathcal X'_{I\setminus i, \underline{a}}),
\end{equation}
where $\zeta$ is the precision parameter in the criterion \eqref{criterion} of bad boxes.
\end{lemma}

\medskip

\medskip

To obtain \eqref{step1}, we need to iterate the result of Lemma \ref{lem1} and get
\begin{equation}\label{iterate}
\mu_{\Delta_L,0,+}^{\nn}(\cup_{I:\, |I|\geq \delta N_{K,L} / 2}\mathcal X'_{I,\underline{a}})
\leq
\sum_{I:\, |I|\geq \delta N_{K,L} /2}(c \zeta^{-2} K^{-d})^{|I|}
\leq
2^{2 N_{K,L}} (c \zeta^{-2} K^{-d})^{\delta N_{K,L} /2},
\end{equation}
which agrees with the one in the right hand side of \eqref{step1} since
$r(\delta, \zeta, K):=-\delta \log (\zeta^{-2}K^{-d})$ is sufficiently large by considering
$K$ large for $\zeta$ and $\delta$ fixed.

\bigskip

To find an estimate for the second contribution in \eqref{split1_5}, 
we use the random-cluster formulation. 
We give a complete description of the method in Appendix \ref{pfoflem2}, where we also provide the proof of the
following lemma:

\begin{lemma}\label{lem2}
Suppose $\beta>\log\sqrt 5$.
For every $\delta>0$, there exists $C=C(\delta)>0$ such that the exponential bound
\begin{align}\label{chi''}
\mu_{\Delta_L,0,\emptyset}^{\nn}(\cup_{I:\, |I|\geq \delta N_{K,L} }\mathcal X''_{I})
\leq e^{-CN_{K,L}}
\end{align}
holds for $K$ (and $L$) large enough.
\end{lemma}
Note that here, for simplicity of the proof, we can use empty boundary conditions by making an extra error of smaller order.
From \eqref{split1_5}, \eqref{iterate} and \eqref{chi''}, we conclude the proof of \eqref{step1}. \qed

\bigskip

\subsection{Proof of \eqref{third}}\label{subsec:third}

When $R\to\infty$, for any translation invariant measure $\mu$ (either $\mu^{\nn}_{0,\pm}$ or $\mu^{\nn}_{h(u)}$ for some $|u|>m_\beta$)
we have that
\begin{equation}\label{lastlimit}
\E_{\mu}[g(m_{B_R})]\to g(\E_{\mu}[\sigma(x)]).
\end{equation}
Similarly, if $R$ depends on $\gamma$ and we pass simultaneously to the limit in such
a way that $1\ll R_\gamma\ll \gamma^{-1}$. \qed

\bigskip

\appendix

\section{Homogeneous magnetization}\label{hom}

For the nearest-neighbour interaction
and
for $h\in{\mathbb R}$, we define the finite volume pressure by
\begin{align}
p_{\Lambda_\eps, \beta}(h) \defi  \frac{1}{\beta\abs{\Lambda_{\eps}}}\log{\sum_{ \sigma\in\Omega_{\Lambda_{\eps}} } e^{-\beta H_{\Lambda_{\eps}, h}\pare{\sigma}} }.
\end{align}
Moreover, for $u\in I_{\abs{\Lambda_\eps}}$, we define the finite volume free energy by
\begin{align}\label{kiwi}
f_{\Lambda_{\eps}, \beta}(u)  \defi  
-\frac{1}{\beta\abs{\Lambda_{\eps}}}\log{\sum_{\substack{ \sigma\in\Omega_{\Lambda_{\eps}} \\[0.1cm] m_{\Lambda_\eps}\pare{\sigma}=u}} e^{-\beta H_{\Lambda_{\eps}}^{nn}\pare{\sigma}} }
\end{align}
and extend the domain of $f_{\Lambda_\eps, \beta}$ to $\corch{-1,1}$ by assigning the values that correspond to linear interpolation
between the values of $f_{\Lambda_{\eps}, \beta}$ at the neighbouring points in $I_{\abs{\Lambda_\eps}}$.
We next prove the existence of the infinite volume free energy and pressure.

\begin{theorem}[Free energy and pressure]\label{tfree} 
The sequence of functions $\pare{f_{\Lambda_\eps,\beta}}_\eps$ 
converges point-wise to a function $f_\beta:\corch{-1,1}\rightarrow{\mathbb R}$ called  free energy.
The function
$f_\beta$ is convex and continuous, differentiable in the interior of its domain. Its derivative $f'_\beta$ is continuous, it satisfies $\lim_{u\downarrow -1}f'_\beta\pare{u}=-\infty$ and $\lim_{u\uparrow 1}f'_\beta\pare{u}=\infty$
and it is bounded on compact subsets of $\pare{-1,1}$.
Moreover, the convergence $\lim_{\eps\rightarrow 0}f_{\Lambda_\eps,\beta}=f_\beta$ is uniform.

Similarly, the sequence of functions $\pare{p_{\Lambda_\eps,\beta}}_\eps$ 
converges point-wise to a function $p_\beta:{\mathbb R}\rightarrow{\mathbb R}$ called pressure and it is given by
\begin{align}\label{hompressure}
p_\beta(h)=\sup_{u\in [-1,1]}\llav{uh-f_\beta(u)},
\end{align}
for every $h\in{\mathbb R}$.
\end{theorem}

\begin{proof}

This is a classical result (see e.g. \cite{E})
with the exception of the uniform convergence of the free energy, which is given here.
With a slight abuse of notation we use $\Lambda_{q} \defi \Lambda_{\eps}$ with $\eps=2^{-q}$.
Observe that $\Lambda_{q+1}$ is the disjoint union of the sub-domains $\Lambda_{q, 1}, \ldots, \Lambda_{q, 2^d}$, each of which is a translation of $\Lambda_{q}$.
For a configuration $\sigma\in\Omega_{\Lambda_{q+1}}$, we call $\sigma_{i}$,
$i=1,\ldots,2^d$ its projections over these sub-domains, i.e., $\sigma=\sigma_1\vee\ldots\vee\sigma_{2^d}$
where by $\vee$ we denote the concatenation on the sub-domains.
Let $u\in I_{\abs{\Lambda_{q}}}$.
Observe that if $m_{\Lambda_{q,i}}(\sigma_i)=u$ for all $i=1,\ldots,2^d$ then $m_{\Lambda_{q+1}}(\sigma)=u$.
Note also that
there are $O\pare{2^d |\partial\Lambda_q|}$ many edges connecting vertices of different sub-domains, where by 
$\partial\Lambda_q$ we denote the boundary of the set.
As a consequence, after defining
\begin{align}
\alpha_q \defi \sum_{\substack{\sigma\in\Omega_{\Lambda_q} \\ m_{\Lambda_q}(\sigma)=u}}
e^{-\beta H_{\Lambda_q}^{\nn}(\sigma)},
\end{align}
we neglect the contributions between the sub-domains, so for some $C>0$ we obtain:
\begin{align}
\alpha_{q+1}\geq
e^{-\beta C 2^d |\partial \Lambda_q|}
\sum_{\substack{ \sigma_{1}\in \Omega_{\Lambda_{q,1}} \\ m_{\Lambda_{q,1}}(\sigma_1)=u }} \ldots
\sum_{\substack{ \sigma_{2^d}\in \Omega_{\Lambda_{q,{2^d}}} \\ m_{\Lambda_{q,2^d}}(\sigma_{2^d})=u }}
\prod_{i=1}^{2^d}e^{-\beta H_{\Lambda_{q,i}}^{\nn}(\sigma_i)}
=e^{-\beta C 2^d |\partial \Lambda_q|} \alpha_q^{2^d}.
\end{align}
Taking logarithm and dividing by $-\beta \abs{\Lambda_{q+1}}$, we get
\begin{align}\label{salmoncita}
f_{\Lambda_{q+1},\beta}\pare{u}\leq f_{\Lambda_{q},\beta}\pare{u}+O\pare{ 2^{-q}}.
\end{align}
For a configuration $\sigma\in\Omega_{\Lambda_{q+1}}$, let $N^+\pare{\sigma} \defi \abs{\llav{x\in\Lambda_{q+1}:\sigma(x)=1}}$ be the associated number of pluses.
There is a correspondence between $I_{\abs{\Lambda_{q+1}}}$ and the set $\corch{0,\abs{\Lambda_{q+1}}}\cap{\mathbb Z}$ containing all possible number of pluses.
Let $u$ and $u'$ be consecutive elements of $I_{\abs{\Lambda_{q+1}}}$ such that $u<u'$, and let
$n$ and $n+1$ be respectively their associated number of pluses.
Then
\begin{align}\label{cabra}
\sum_{\substack{ \sigma\in\Omega_{\Lambda_{q+1}} \\[0.05cm] m\pare{\sigma}=u' }} e^{-\beta H_{\Lambda_{{q+1}}}^{\nn}\pare{\sigma}}=
\sum_{\substack{ \sigma\in\Omega_{\Lambda_{{q+1}}} \\[0.05cm] N^+\pare{\sigma}=n+1 }} e^{-\beta H_{\Lambda_{{q+1}}}^{\nn}\pare{\sigma}}
=\frac{1}{n+1}
\sum_{\substack{ \sigma\in\Omega_{\Lambda_{{q+1}}} \\[0.05cm] N^+\pare{\sigma}=n }}
\, \sum_{\substack{ \sigma'\in\Omega_{\Lambda_{{q+1}}} \\[0.05cm] \sigma'\geq \sigma
\\[0.05cm] N^+\pare{\sigma'}=n+1  }}
e^{-\beta H_{\Lambda_{{q+1}}}^{\nn}\pare{\sigma'}},
\end{align}
where $\sigma'\geq \sigma$ means $\sigma'(x)\geq\sigma(x)$ for every $x\in\Lambda_{q+1}$.
In the later sum, the configurations $\sigma'$ differ from $\sigma$ just in one site.
As every site has $2^d$ neighbours, we have
\begin{align}\label{jarra}
H_{\Lambda_{q+1}}^{\nn}\pare{\sigma'}\geq H_{\Lambda_{q+1}}^{\nn}\pare{\sigma}-2^{d+1}.
\end{align}
Replacing in \eqref{cabra}, using the fact that  $\abs{\llav{\sigma':\sigma'\geq \sigma,N^+\pare{\sigma'}=n+1}}=\abs{\Lambda_{q+1}}-n$ 
(i.e., the number of minuses in the $\sigma$ configuration)
and the bound $\frac{\abs{\Lambda_{q+1}}-n}{n+1}\leq \abs{\Lambda_{q+1}}$, we get
\begin{align}
\sum_{\substack{ \sigma\in\Omega_{\Lambda_{q+1}} \\[0.05cm] m\pare{\sigma}=u' }} e^{-\beta H_{\Lambda_{q+1}}^{\nn}\pare{\sigma}}
\leq
\abs{\Lambda_{q+1}}e^{\beta 2^{d+1}} \sum_{\substack{ \sigma\in\Omega_{\Lambda_{q+1}} \\[0.05cm] m\pare{\sigma}=u }} e^{-\beta H_{\Lambda_{q+1}}^{\nn}\pare{\sigma}}.
\end{align}
Taking logarithm and dividing by $-\beta\abs{\Lambda_{q+1}}$, we get
\begin{align}
f_{\Lambda_{q+1},\beta}\pare{u}-f_{\Lambda_{q+1},\beta}\pare{u'}\leq
O\pare{\frac{\log\abs{\Lambda_{q+1}}}{\abs{\Lambda_{q+1}}}}.  
\end{align}
For $f_{\Lambda_{q+1},\beta}\pare{u'}-f_{\Lambda_{q+1},\beta}\pare{u}$
the same bound can be obtained   by replacing the number of pluses $N^+$ by the number of minuses $N^-$.
Then
\begin{align}\label{salmoncito}
\abs{f_{\Lambda_{q+1},\beta}\pare{u}-f_{\Lambda_{q+1},\beta}\pare{u'}}=
O\pare{\frac{\log\abs{\Lambda_{q+1}}}{\abs{\Lambda_{q+1}}}}.
\end{align}

For $u\in I_{\abs{\Lambda_{q+1}}}$, let $u_{-}$ and $u_{+}$ be the best approximates 
in $I_{\abs{\Lambda_{q}}}$ of $u$ from below and from above:
\begin{align}
u_- \defi  \max\llav{m'\in I_{\abs{\Lambda_{q}}}:m'\leq u},
\ \ \ \ \ \ 
u_+ \defi  \min\llav{m'\in I_{\abs{\Lambda_{q}}}:m'\geq u}.
\end{align}
For $u\in I_{\abs{\Lambda_{q+1}}}\setminus I_{\abs{\Lambda_{q}}}$, using \eqref{salmoncito} repeatedly, we get
\begin{eqnarray}
f_{\Lambda_{q+1},\beta}\pare{u} & \leq &
f_{\Lambda_{q+1},\beta}\pare{u_-}\wedge f_{\Lambda_{q+1},\beta}\pare{u_+}+
O\pare{\frac{\log\abs{\Lambda_{q+1}}}{\abs{\Lambda_{q+1}}}}\nonumber\\
& \leq &
f_{\Lambda_{q},\beta}\pare{u}+
O\pare{\frac{\log\abs{\Lambda_{q+1}}}{\abs{\Lambda_{q+1}}}},
\end{eqnarray}
after using \eqref{salmoncita}
and the fact that $2^{-q}\ll \frac{\log\abs{\Lambda_{q+1}}}{\abs{\Lambda_{q+1}}}$.

Let $a_q \defi  O\pare{\frac{\log\abs{\Lambda_{q+1}}}{\abs{\Lambda_{q+1}}}}$ and observe that $a \defi \sum_qa_q$ is finite.
From the above estimates and the fact that $f_{\Lambda_q,\beta}$ is defined by linear interpolation, we obtain
\begin{align}\label{opus}
f_{\Lambda_{q+1},\beta}\pare{u}\leq f_{\Lambda_{q},\beta}\pare{u}+a_q,
\end{align}
for every $u\in\corch{-1,1}$ and every $q$.
Let $g_{q} \defi  f_{\Lambda_q,\beta}-\sum_{i=0}^{q-1} a_i$.
Inequality \eqref{opus} implies that
\begin{align}\label{aje}
g_{{q+1}}(u)\leq g_{q}(u),
\end{align}
for every $u\in \corch{-1,1}$.
The point-wise convergence of the free energy
guarantees the 
point-wise convergence of $\pare{g_{q}}_q$ to $f_\beta-a$.
Then $(g_{q})_q$ is a sequence of continuous functions defined on a compact set that converges point-wise and in a monotonic way to $f_\beta-a$.
Under these hypotheses, Dini's theorem asserts that the convergence is uniform,
hence concluding the uniform convergence of $f_{\Lambda_q,\beta}$ to $f_\beta$.
\end{proof}

\bigskip

\subsection{Proof of Lemma \ref{AA}}\label{add}

We consider identity \eqref{cabra} with $H_{\Lambda_q}^{\nn}$ replaced by $H_{\Lambda_{q},\gamma, \alpha}$.
While comparing $H_{\Lambda_{q},\gamma, \alpha}\pare{\sigma}$ with $H_{\Lambda_{q},\gamma, \alpha}\pare{\sigma'}$,
the nearest-neighbour part of the Hamiltonian can be treated as in the proof of Theorem \ref{tfree}.
To treat  the quadratic part, observe that, as every vertex interacts with $O\pare{\gamma^{-d}}$ vertices, we have
\begin{align}
H_{\Lambda_{q},\gamma, \alpha}\pare{\sigma'}\geq H_{\Lambda_{q},\gamma, \alpha}\pare{\sigma} + O\pare{\gamma^{-d}}.
\end{align}
We can now repeat the arguments of the proof of Theorem \ref{tfree}: using \eqref{salmoncito}
and \eqref{salmoncita} with error $O(2^{-q}\gamma^{-d})$,
for $t$ and $t'$ consecutive elements of $I_{|\Lambda_q|}$ we have that
\begin{align}
\abs{F_{\Lambda_q, \gamma,\alpha}(t)
-F_{\Lambda_q, \gamma,\alpha}(t')}
 \leq
 C 2^{-q}\gamma^{-d}
 +\frac{\log\abs{\Lambda_q}}{\abs{\Lambda_q}},
\end{align}
where $C$ is a constant that depends only on the dimension. \qed

\bigskip

\subsection{Proof of Lemma \ref{pedro}}\label{lem51}

Given $c,\zeta>0$, let $B_{\Lambda_\eps,\gamma,\zeta, >c}$ be the set of mostly bad spin configurations:
\begin{equation}
B_{\Lambda_\eps,\gamma,\zeta, >c} \defi  \Big\{ \sigma \in \Omega_{\Lambda_\eps}:\ |\{x \in \Lambda_\eps:\ |I_x^\gamma(\sigma)-u|>\zeta \}| >c |\Lambda_\eps| \Big\} ,
\end{equation}
where $I_x^\gamma$ is defined in \eqref{Igamma}.
Let $u\in\corch{-m_\beta,m_\beta}$.
To get an upper bound of
\begin{equation}\label{hola00}
\mu_{\Lambda_\eps,\gamma,\alpha}(B_{\Lambda_\eps,\gamma,\zeta, >c})=
\frac{1}{Z_{\Lambda_{\eps},\gamma,\alpha}}
\sum_{\sigma\in B_{\Lambda_\eps,\gamma,\zeta, >c}}e^{-\beta H_{\Lambda_{\eps},\gamma,\alpha}\pare{\sigma}} ,
\end{equation}
we look for a lower bound for the Kac part of the Hamiltonian.
Indeed, for $\sigma\in B_{\Lambda_\eps,\gamma,\zeta, >c}$ 
since $\alpha=u$ (from \eqref{EL}) we have that
\begin{align}\label{hola01}
\sum_{x\in\Lambda_\eps}\pare{I_x^{\gamma}\pare{\sigma}-u}^2\geq c\zeta^2\abs{\Lambda_\eps},
\end{align}
which further implies that
\begin{align}\label{hola11}
\sum_{\sigma\in B_{\Lambda_\eps,\gamma,\zeta, >c}}e^{-\beta H_{\Lambda_{\eps},\gamma,\alpha}\pare{\sigma}} 
\leq
e^{-\beta c \abs{\Lambda_\eps}\zeta^2}\sum_{\sigma\in \Omega_{\Lambda_\eps}}e^{-\beta H_{\Lambda_\eps}^{\nn}\pare{\sigma}}.
\end{align}
Furthermore, from Theorem \ref{tfree}  we have that there exists an error $s_1(\eps)\to 0$ as $\eps\to 0$
such that 
\begin{align}\label{hola22}
\sum_{\sigma\in \Omega_{\Lambda_\eps}}e^{-\beta H_{\Lambda_\eps}^{\nn}\pare{\sigma}}
=e^{-\beta\abs{\Lambda_\eps}
\corch{p_\beta(0)+s_1(\eps)}}.
\end{align}
On the other hand, to estimate the denominator of \eqref{hola00}, note that 
since $u\in [-m_\beta,m_\beta]$, we have that
$f'_\beta\pare{u}=0$. 
Then, \eqref{pressure} gives $P\pare{\alpha}=-f_\beta\pare{u}$ for $u=\alpha$; hence,
\begin{align}\label{hola44}
Z_{\Lambda_{\eps},\gamma,\alpha}=
e^{-\beta\abs{\Lambda_\eps}\corch{f_\beta\pare{u}+s_2\pare{\eps,\gamma}}},
\end{align}
for some error $s_2\pare{\eps,\gamma}$ vanishing as $\eps$ and $\gamma$ go to zero.
Thus, if we substitute \eqref{hola11}, \eqref{hola22} and \eqref{hola44} into \eqref{hola00}, 
we obtain that:
\begin{align}
\mu_{\Lambda_\eps,\gamma,\alpha}(B_{\Lambda_\eps,\gamma,\zeta, >c})
\leq
e^{-\beta\abs{\Lambda_\eps}\corch{c\zeta^2+s_1\pare{\eps}-s_2\pare{\eps,\gamma}}}
\leq
e^{-\beta\abs{\Lambda_\eps}c\zeta^2/2},
\end{align}
for $\eps$ and $\gamma$ small enough.

\medskip

If $u\notin\corch{-m_\beta,m_\beta}$ we have a similar strategy but for the appropriate external field.
Hence, adding and subtracting $u$ we expand the Hamiltonian as follows:
\begin{equation}\label{gotoh}
H_{\Lambda_{\eps},\gamma,\alpha}\pare{\sigma}=
H_{\Lambda_\eps}^{\nn}\pare{\sigma}+
\sum_{x\in\Lambda_\eps}\corch{I_x^\gamma\pare{\sigma}-u}^2+
2\pare{u-\alpha}\sum_{x\in\Lambda_\eps}{I_x^\gamma\pare{\sigma}}
-2u\pare{u-\alpha}\abs{\Lambda_\eps}+\pare{u-\alpha}^2\abs{\Lambda_\eps}.
\end{equation}
Note that for the computation of \eqref{hola00} the constant terms are irrelevant.
In the case $u\notin\corch{-m_\beta,m_\beta}$, from \eqref{EL} we have that 
\begin{align}
f'_\beta\pare{u}=-2\pare{u-\alpha};
\end{align}
hence
our goal is to approximate the Hamiltonian $H_{\Lambda_{\eps},\gamma,\alpha}$ by \eqref{guante} with external field $h \defi   -2\pare{u-\alpha}$. 
Additionally, recalling \eqref{Iandone} and \eqref{main_choice_ofL} we obtain that
\begin{align}
\sum_{y\in\Lambda_\eps}I_y^{\gamma}\pare{\sigma}=
\sum_{y\in\Lambda_\eps}\sum_{x\in\Lambda_\eps}J_{\gamma}\pare{x,y}\sigma(x)=
\sum_{x\in\Lambda_\eps}\sigma(x)\sum_{y\in\Lambda_\eps}
J_{\gamma}\pare{x,y}=
\sum_{x\in\Lambda_\eps}\sigma(x)+s\pare{\gamma}O\pare{1}\abs{\Lambda_\eps},
\end{align}
for some $s(\gamma)\to 0$ as $\gamma\to 0$.
Then, with $h$ defined above, the third term of \eqref{gotoh} gives $-h \sum_{x\in \Lambda_\eps}\sigma(x)$
with a vanishing error.
For the second term we use \eqref{hola01}.

For the denominator, we restrict to 
\begin{equation}
B_{\Lambda_\eps,\gamma,\zeta',\leq c'} \defi 
\Big\{ \sigma \in \Omega_{\Lambda_\eps}:\ |\{x \in \Lambda_\eps:\ |I_x^\gamma(\sigma)-u|>\zeta' \}| \leq c' |\Lambda_\eps| \Big\} 
\end{equation}
for $c', \zeta'>0$ to be chosen appropriately.
Then for $\sigma\in B_{\Lambda_\eps,\gamma,\zeta',\leq c'}$ we have that
\begin{align}
\sum_{x\in\Lambda_\eps}\pare{I_x^{\gamma}\pare{\sigma}-u}^2\leq
O\pare{1}c'\abs{\Lambda_\eps}+\zeta'^2\abs{\Lambda_\eps}.
\end{align}
Thus, replacing all above estimates in  \eqref{hola00}, with $c', \zeta'>0$ 
chosen such that $O\pare{1}s\pare{\gamma}+O\pare{1}c'+\zeta'^2$ is smaller than $c\zeta^{2}$
(also $\gamma$ small enough)
we obtain:
\begin{align}
\mu_{\Lambda_{\eps}, \beta, \gamma,\alpha}\pare{B_{\Lambda_\eps,\gamma,\zeta,>c}}\leq
\frac{e^{-\beta\abs{\Lambda_\eps}\pare{c\zeta^2+O\pare{1}s\pare{\gamma}}}}{
e^{-\beta\abs{\Lambda_\eps}\pare{O\pare{1} c'+\zeta'^2 }}
\mu^{\nn}_{\Lambda_\eps, h}
\pare{B_{\Lambda_\eps,\gamma,\zeta',\leq c'}}}
\leq
\frac{e^{-C\beta\abs{\Lambda_\eps}}}{\mu^{\nn}_{\Lambda_\eps, h}
\pare{B_{\Lambda_\eps,\gamma,\zeta',\leq c'}}},
\end{align}
for some $C>0$.
Thus, to conclude it suffices to prove that the denominator is close to $1$. This is the content of
the next lemma:

\medskip

\begin{lemma}\label{ser}
Let $u\in\corch{-1,1}\setminus\corch{-m_\beta,m_\beta}$ and let $h=f_\beta'(u)$ be the external field that corresponds to the homogeneous magnetization $u$.
Then, for the measure \eqref{Ising} we have that
\begin{align}
\lim_{\gamma\rightarrow 0}\lim_{\eps\rightarrow 0}\mu^{\nn}_{\Lambda_\eps, h}
(B_{\Lambda_\eps,\gamma,\zeta, >c})=0
\end{align}
for every $\zeta, c>0$.
\end{lemma}

\medskip

\begin{proof}[Proof of lemma \ref{ser}.]
For $\sigma\in B_{\Lambda_\eps,\gamma,\zeta, >c}$, we have
\begin{align}
\frac{1}{\abs{\Lambda_\eps}}\sum_{x\in\Lambda_\eps}\abs{I_x^{\gamma}\pare{\sigma}-u}\geq c\zeta,
\end{align}
which implies that
\begin{align}\label{sabor}
\mathbb{E}_{\mu^{\nn}_{\Lambda_\eps,h}}\left[ \frac{1}{\abs{\Lambda_\eps}} \sum_{x \in \Lambda_\eps} \abs{I_x^{\gamma}-u}  \right]
\geq c\zeta \, \mu^{\nn}_{\Lambda_\eps,h}\pare{B_{\Lambda_\eps,\gamma,\zeta, >c}}.
\end{align}
By translation invariance of the measure $\mu^{\nn}_{\Lambda_\eps,h}$ 
with periodic boundary conditions, the left-hand side of \eqref{sabor} coincides with
$\mathbb{E}_{\mu^{\nn}_{\Lambda_\eps,h}}\left[ \abs{I_0^{\gamma}-u}  \right]$. 
Since the random variable $\abs{I_0^{\gamma}-u}$ depends on a finite number of coordinates, 
the later expectation converges to 
$ {\mathbb E}_{\mu^{\nn}_{h}} \left[ \abs{I_0^{\gamma}-u}\right]$ as $\eps\to 0$.
Note that $\mu^{\nn}_h$ is the infinite volume limit of \eqref{Ising}.	
Then, by applying the multidimensional ergodic theorem (e.g. Theorem 14.A8 of \cite{Geo11}),
we obtain that
\begin{align}
\lim_{\gamma\rightarrow 0}\,
{\mathbb E}_{\mu^{\nn}_{h}} \left[ \abs{I_0^{\gamma}-u}\right]=0,
\end{align}
since $\E_{\mu^{\nn}_{h}}[\sigma(x)]=u$ for all $x\in B_{\gamma^{-1}}(0)$ and $\gamma^{-1}\to\infty$.
\end{proof}

\bigskip

\section{Estimates on ``bad boxes".}\label{est_baC_box}\label{appB}

Before proceeding with the estimates on ``bad boxes'',
we state a theorem for the infinite volume Gibbs measures for the Ising model:

\begin{theorem}\label{factsgibbs}
For $d\geq 2$, $h=0$ and $\beta>\beta_c(d)$ ($\beta_c(d)$ is the critical value of the inverse temperature in dimension $d$), there are two different probability measures $\mu^{\nn}_{0,\pm}$
on $\{-1,1\}^{\mathbb Z^d}$ such that,
for any sequence of increasing volumes $(\Lambda_n)_n$, 
the sequence 
$\mu^{\nn}_{\Lambda_\eps, 0,\pm}$ (with $\pm$ boundary conditions) converges
weakly to $\mu^{\nn}_{0,\pm}$.
More precisely, for $\Delta\subset \Lambda$ finite subsets of $\mathbb Z^d$ and $f:\{-1,1\}^{\mathbb Z^d}\rightarrow\mathbb R$ a function that depends only on spins inside $\Delta$, there exists a positive constant $C$ such that 
\begin{equation}\label{prel1}
\Big| \E_{\mu^{\nn}_{\Lambda, 0,+}}[f]-\E_{\mu^{\nn}_{0,+}}[f]\Big|\leq
C |\Lambda| e^{-\beta \text{dist}(\Delta,\Lambda^c)}.
\end{equation}
Furthermore, exponential decay of correlations holds: if the functions $f$ and $g$ depend on spins inside the finite regions $\Delta_f$ and $\Delta_g$, respectively,
then there exist positive constants $C_1$ and $C_2$ such that
\begin{equation}\label{prel2}
\Big| \E_{\mu^{\nn}_{0,+}}[f g]-\E_{\mu^{\nn}_{0,+}}[f]\E_{\mu^{\nn}_{0,+}}[g]\Big|\leq
C_1 e^{-C_2 \text{dist}(\Delta_f,\,\Delta_g)}.
\end{equation}
\end{theorem}

The proof is standard and can be found in \cite{P15}, Theorem 2.5 and its proof in Section 2.6.2. See also Theorem 2.18.

\subsection{Proof of Lemma \ref{lem1}.}\label{pfoflem1}
In this section we give the following proof:

\bigskip

{\it Proof of Lemma \ref{lem1}:}
Let $\Delta_K^{00}$ be the cube with the same center as $\Delta_K^0$ and at distance $K^{\frac 12}$
from its complement.
Given a set $S\subset \Lambda_\eps$, an accuracy parameter $\zeta$ and a radius $R>0$ we define the following set of configurations:
\begin{equation}\label{D}
C_{S,\zeta} \defi \Big\{
\sigma:
\Big|
\sum_{x\in S}g(m_{B_R(x)})
-
\sum_{x\in S}
\E_{\mu^{\nn}_{0,\tau}}[g(m_{B_R(x)})]
\Big|
\geq \zeta |S|
\Big\}.
\end{equation}
We have that for any $\tau\in\{+,-\}$ and for $L$ large enough
\begin{equation}\label{subset}
C_{\Delta_K^0,\zeta}\subset C_{\Delta_K^{00},\frac{\zeta}{2}}.
\end{equation}
Given $i\in I$ and $C\in\mathcal K_i$ the sets $G_{C,i}$ and $\mathcal X'_{I\setminus i}$ are $C^c$ measurable
while the set $C_{\Delta_{K,i}^0,\zeta}$ is $C$ measurable. Hence, using \eqref{subset} we obtain
\begin{eqnarray}\label{split2}
\mu^{\nn}_{\Delta_L, 0,+}(\mathcal X'_I)
&=&
\sum_{C\in\mathcal K_i}
\mu^{\nn}_{\Delta_L, 0,+}(G_{C,i}\cap C_{\Delta_{K,i}^0,\zeta}\cap\mathcal X'_{I\setminus i})\nonumber\\
&\leq&
\sum_{C\in\mathcal K_i}
\mu^{\nn}_{C, 0,+}(C_{\Delta_{K,i}^{00},\frac{\zeta}{2}}) 
\mu^{\nn}_{C^c, 0,+}(G_{C,i}\cap \mathcal X'_{I\setminus i}),
\end{eqnarray}
where we have used the restricted measures $\mu^{\nn}_{C, 0,+}$ and 
$\mu^{\nn}_{C^c, 0,+}$ instead
of $\mu^{\nn}_{\Delta_L, 0,+}$.
From the exponential decay of correlations \eqref{prel2}, we have that there are two positive constants $C_1$ and $C_2$ such that
\begin{equation}\label{correlations}
\E_{\mu^{\nn}_{C, 0,\tau}}
\Big[
g(m_{B_R(x)})g(m_{B_R(y)})-
\E_{\mu^{\nn}_{0,\tau}}[g(m_{B_R(x)})] \E_{\mu^{\nn}_{0,\tau}}[g(m_{B_R(y)})]
\Big]
\leq
C_1
e^{- C_2 |x-y|}.
\end{equation}
Then, using the Chebyshev inequality and the weak convergence to an infinite volume limit \eqref{prel1}
we obtain:
\begin{eqnarray}\label{Cheb}
\mu^{\nn}_{C, 0,+}(C_{\Delta_{K,i}^{00},\frac{\zeta}{2}}) 
& \leq &
\Big(\frac{1}{ \zeta |\Delta_{K,i}^{00}|}\Big)^2
\sum_{x,y\in \Delta_{K,i}^{00}}
\E_{\mu^{\nn}_{C, 0,+}}
\Big[
\prod_{z=x,y}
\Big(
g(m_{B_R(z)})
-
\E_{\mu^{\nn}_{0,+}}[g(m_{B_R(z)})]
\Big)
\Big]
\nonumber\\
&\leq &
c \zeta^{-2} K^{-d},
\end{eqnarray}
for some $c>0$ and
where $R$ is such that $R\ll K^{1/2}$.
Then from \eqref{split2} we obtain:
\begin{equation}\label{final}
\mu^{\nn}_{\Delta_L, 0,+}(\mathcal X'_I)
\leq c \zeta^{-2} K^{-d}
\sum_{C\in\mathcal K_i}
\mu^{\nn}_{C^c,0,+}(G_{C,i}\cap \mathcal X'_{I\setminus i}),
\end{equation}
which gives the right hand side of \eqref{chi'} by using the fact that the events $G_{C,i}$ 
for $C\in\mathcal K_i$ are disjoint.
\qed

\medskip

\subsection{Proof of Lemma \ref{lem2}}\label{pfoflem2}

For completeness of the presentation, we first give a short description of the method
and then proceed with the proof of the relevant Lemma \ref{lem2}.
We restrict ourselves to dimension $2$, but we expect that
a similar result should be also true in higher dimensions.
As before, we divide $\Delta_L$ into boxes $\Delta_{K,i}$, and call $N_{K,L}=\frac{L^2}{K^2}$.
Recall that $\Delta_{K,i}^0$
stands for the box with the same center as $\Delta_{K,i}$ and distance $\sqrt K$
from its complement $C^c_{K}$.

Let $E(\Delta_L)$ be the set of edges connecting vertices in $\Delta_L$:
$E(\Delta_L)\defi\llav{\llav{x,y}\subset \Lambda_L:\abs{x-y}=1}$.
The random-cluster probability for $\om\in \llav{0,1}^{E(\Delta_L)}$ is defined by
\begin{align}\label{phi}
\phi(\om) \defi \frac{1}{Z'}\Big\{\prod_{\pic{xy}\in E}p^{\om_{\pic{xy}}}(1-p)^{1-\om_{\pic{xy}}}\Big\}2^{\textnormal{Cl}(\om)},
\end{align}
where $p \defi 1-e^{-2\beta}$, $\textnormal{Cl}(\om)$ is the number of connected components (or clusters) associated to $\om$, and $Z'$ is the normalizing constant.
The Edwards-Sokal probability $Q$, see \cite{ES}, 
is defined on the product space $\llav{0,1}^{E\pare{\Delta_L}}\times\llav{-1,1}^{\Delta_L}$ and has the random-cluster probability $\phi$ as the first marginal and the Ising probability $\mu_{\Delta_L,0,\empty}$ as the second marginal.
The main property of $Q$ is that the conditional probability $Q(\cdot|\om)$ is given by sampling a value of a spin independently in each cluster of $\om$ with probability $\frac{1}{2}$.
In this way, if $x,y\in \Delta_L$ and $\om\in\llav{0,1}^{E(\Delta_L)}$ are 
such that $x$ and $y$ are connected by a path of edges $e_1,\ldots,e_k$ such that $\om_{e_i}=1$ for every $i$, then
$Q(\sigma(x)=\sigma(y)|\om)=1$.

We can define a partial order on the probability space $\llav{0,1}^{E(\Delta_L)}$ by $\om\preccurlyeq\om'$ if and only if $\om_e\leq \om'_e$ for every $e\in E(\Delta_L)$.
A function $f:\llav{0,1}^{E(\Delta_L)}\rightarrow {\mathbb R}$ is increasing (resp. decreasing) if and only if $f(\om)\leq f(\om')$ (resp. $f(\om)\geq f(\om')$) for every $\om,\om'$ such that $\om\preccurlyeq\om'$;
an event $A\subset \llav{0,1}^{E(\Delta_L)}$ is increasing (resp. decreasing) if the indicator function $\textbf{1}_A$ is an increasing (resp. decreasing) function.
For probabilities $P$ and $P'$ on $\llav{0,1}^{E(\Delta_L)}$, we say that $P$ is stochastically dominated by $P'$, and write $P\leq_{\textnormal{st}}P'$, if and only if $\int f dP\leq \int f dP'$ for every increasing function $f$.
The later property holds if and only if $\int f dP\geq \int f dP'$ for every decreasing function $f$.
Let $B_\rho$ be the Bernoulli probability on $\llav{0,1}^{E(\Delta_L)}$ with parameter  $\rho \defi \frac{1-e^{-2\beta}}{1+e^{-2\beta}}$:
\begin{align}\label{bern}
B_\rho(\om) \defi \prod_{e\in E}\rho^{\om_{e}}(1-\rho)^{1-\om_{e}}.
\end{align}
The random-cluster probability satisfies $B_\rho\leq_{st} \phi$;
in particular, $\phi(A)\leq B_\rho(A)$ for every decreasing event $A$.
We are ready now to give the proof of Lemma \ref{lem2}.

\bigskip
{\it Proof of Lemma \ref{lem2}:} We need to introduce some terminology.
Let ${\mathbb Z}^{2*} \defi {\mathbb Z}^2+\pare{\frac{1}{2},\frac{1}{2}}$ be the dual set of vertices of ${\mathbb Z}^2$.
For an edge $e=\pic{xy}\in E({\mathbb Z}^2)$, where $E(\mathbb Z^2)$ is the set of edges of  $\mathbb Z^2$, 
we define its dual edge $e^*$ as the one obtained after rotating it $90$ degrees around its middle point;
for an edge subset $A\subset E({\mathbb Z}^2)$, we define $A^* \defi \llav{e^*:e\in A}$.
For any subset of edges $E$, let the support of $E$ be the set of vertices that are extreme vertices of any of the edges in $E$. 
For a subset $R\subset \Delta_L$,
we define
its dual set of vertices $R^*$ as the support of $E(R)^*$.
The inner boundary of $R^*$ is defined by $\partial^{\circ}R^*\llav{x\in R^*:\abs{x-y}=1\textnormal{ for some }y\in{\mathbb Z}^{2*}\setminus R^*}$.
For a configuration $\om\in\llav{0,1}^{E(R)}$, we define its dual configuration $\om^*\in\llav{0,1}^{E(R)^*}$ by $\om^*_{e^*}=1-\om_{e}$.
We say that an edge $e^*\in E(R)^*$ is {\bf $\om^*$-open} if $\om^*_{e^*}=1$.
Associated to $\om^*$, and for a fixed box $\Delta_{K,i}$, we call $J_{\Delta_{K,i}}(\om^*)\subset E(\Delta_{K,i})^*$ 
the set of edges ``penetrating from the outside of $\Delta_{K,i}$", 
i.e., those containing the dual edges that are $\om^*$-open and are connected to $\partial^\circ (\Delta_{K,i}^*)$ by a path of $\om^*$-open edges. 
We say that $\om\in\llav{0,1}^{E(\Delta_L)}$ has a circuit of open edges in $\Delta_{K,i}$ if  $J_{\Delta_{K,i}}(\om^*)\cap E(\Delta_{K,i}^0)^*=\emptyset$ (this is the formal way of saying that $\omega$ has a self-avoiding path of open edges living in $E(\Delta_{K,i})$ that surrounds $\Delta_{K,i}^0$).

Consider the random-cluster probability $\phi$ associated to $\mu_{\Delta_L,0,\empty}$ (defined in \eqref{phi}) and  the corresponding
Edwards-Sokal coupling  $Q$ between $\phi$ and $\mu_{\Delta_L,0,\empty}$.
The fundamental property of $Q$ implies that, for every $\Delta_{K,i}$,
\begin{align}
&Q \Big(\{  (\om,\sigma)\in\llav{0,1}^{E(\Delta_L)}\times \llav{-1,1}^{\Delta_L}:
\\[0.3cm]
& \hspace{1cm} 
\om\textnormal{ has a circuit of open edges in }\Delta_{K,i}, \Delta_{K,i}\textnormal{ is bad of type II} \}\Big)=0.
\end{align}
As a consequence, if we define
 $\mathcal Y''_I$ to be the set of configurations $\om\in\llav{0,1}^{E(\Delta_L)}$ that do not have a circuit of open edges for every $i\in I$,
we have
$\mu_{\Delta_L,0,\emptyset}(\cup_{I:\, |I|\geq \delta N_{K,L} }\mathcal X''_{I})\leq \phi(\cup_{I:\, |I|\geq \delta N_{K,L} }\mathcal Y''_{I})
$;
to conclude, we need to control this last term. 
Recall the Bernoulli probability $B_\rho$ on $\llav{0,1}^{E(\Delta_L)}$, given in \eqref{bern}.
As $\cup_{I:\, |I|\geq \delta  N_{K,L} }\mathcal Y''_{I}$ is a decreasing event, the stochastic domination $B_\rho \leq_{st}\phi$ implies that
$\phi(\cup_{I:\, |I|\geq \delta  N_{K,L} }\mathcal Y''_{I})\leq B_\rho(\cup_{I:\, |I|\geq\delta  N_{K,L} }\mathcal Y''_{I})$.
Observe that inequality
\begin{eqnarray}\label{rr44}
B_{\rho}(\cup_{I:\, |I|\geq \delta N_{K,L} }\mathcal Y''_{I}) \leq
\binom{N_{K,L}}{\lceil \delta N_{K,L}\rceil}
\left(B_\rho(\{\om\in \llav{0,1}^{E(\Delta_K)}:\om\textnormal{ does not have a circuit}\})\right)^{\delta N_{K,L}} \hspace{0.3cm}
\end{eqnarray}
holds, where $\Delta_K$ is any of the boxes $\Delta_{K,1},\ldots,\Delta_{K,N_{K,L}}$.
Moreover, by Stirling's formula, there is a constant $C_1=C_1(\delta)>0$ such that $\binom{N_{K,L}}{\lceil \delta N_{K,L}\rceil}\leq C_1^{N_{K,L}}$
for every $N_{K,L}$.
To estimate 
\begin{align}
B_\rho(\{\om\in \llav{0,1}^{E(\Delta_{K})}:\om\textnormal{ does not have a circuit}\}),
\end{align}
we consider its complement.
Let $R_1$, $R_2$, $R_3$ and $R_4$ be the rectangles of dimension $K\times\frac{K-K^{\frac 12}}{2}$ or $\frac{K-K^{\frac 12}}{2}\times K$ that satisfy $\cup_{i=1}^4 R_i=\Delta_K\setminus \Delta_K^0$ (see Figure \ref{figure1}).
\begin{figure}
\centering
\begin{tikzpicture}[line cap=round,line join=round,>=triangle 45,x=1cm,y=1cm]
\path [fill=gray!30!] (-1,-1.6) -- (-.3,-1.6) -- (-.3,1.6) -- (-1,1.6);
\path [fill=gray!30!] (-7.3,.9) -- (-4.1,.9) -- (-4.1,1.6) -- (-7.3,1.6);
\path [fill=gray!30!] (4.1,-1.6) -- (4.8,-1.6) -- (4.8,1.6) -- (4.1,1.6);
\path [fill=gray!30!] (.3,-.9) -- (3.5,-.9) -- (3.5,-1.6) -- (.3,-1.6);
\draw [line width=1.2pt] (-7.3,-1.6) -- (-4.1,-1.6) -- (-4.1,1.6) -- (-7.3,1.6) -- (-7.3,-1.6);
\draw [line width=1.2pt] (-3.5,-1.6) -- (-.3,-1.6) -- (-.3,1.6) -- (-3.5,1.6) -- (-3.5,-1.6);
\draw [line width=1.2pt] (7.3,-1.6) -- (4.1,-1.6) -- (4.1,1.6) -- (7.3,1.6) -- (7.3,-1.6);
\draw [line width=1.2pt] (3.5,-1.6) -- (.3,-1.6) -- (.3,1.6) -- (3.5,1.6) -- (3.5,-1.6);
\draw [line width=1.2pt] (-7.3,.9) -- (-4.1,.9);
\draw [line width=1.2pt] (-1,1.6) -- (-1,-1.6);
\draw [line width=1.2pt] (.3,-.9) -- (3.5,-.9);
\draw [line width=1.2pt] (4.8,-1.6) -- (4.8,1.6);
\draw (-6.1,1.55) node[anchor=north west] {$R_1$};
\draw (-1.05,0.35) node[anchor=north west] {$R_2$};
\draw (1.5,-.95) node[anchor=north west] {$R_3$};
\draw (4.1,.35) node[anchor=north west] {$R_4$};
\end{tikzpicture}
\caption{The rectangles $R_1$, $R_2$, $R_3$ and $R_4$.} \label{figure1}
\end{figure}
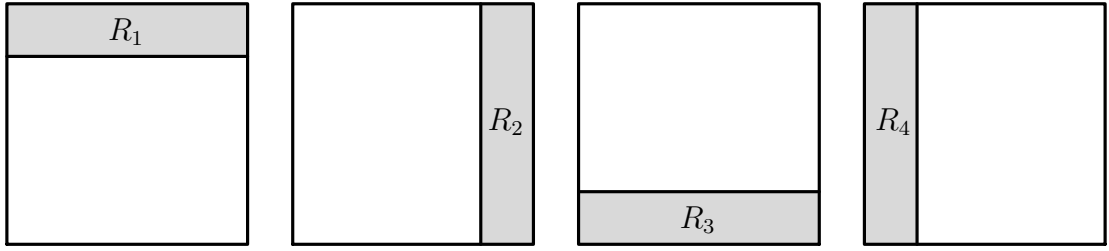

Let $R$ be one of these rectangles and, without loss of generality, suppose it to be horizontal, that is of dimension $K\times\frac{K-K^{\frac{1}{2}}}{2}$.
Let $T(R^*)$ and $B(R^*)$ be the corresponding vertices in the top and in the bottom of the support of the (dual) set of edges $R^*$, that is, the ones with highest and lowest second coordinate (see Figure \ref{figure2}).
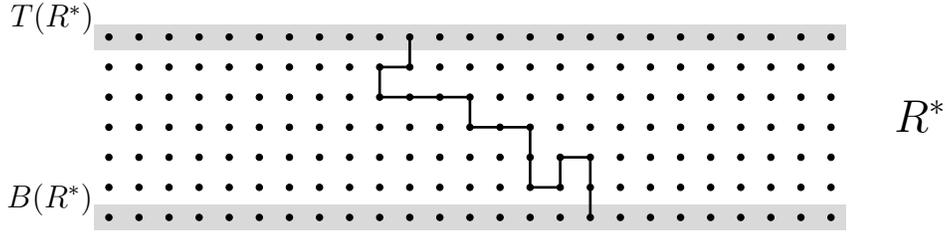
\begin{figure}
\centering
\begin{tikzpicture}[line cap=round,line join=round,>=triangle 45,x=1cm,y=1cm]
\path [fill=gray!30!] (-5,1.37) -- (5,1.37) -- (5,1.03) -- (-5,1.03);
\path [fill=gray!30!] (-5,-1.37) -- (5,-1.37) -- (5,-1.03) -- (-5,-1.03);
\draw[fill]   (0,0) circle (1.2pt);
\draw[fill]   (0.4,0) circle (1.2pt);
\draw[fill]   (0.8,0) circle (1.2pt);
\draw[fill]   (1.2,0) circle (1.2pt);
\draw[fill]   (1.6,0) circle (1.2pt);
\draw[fill]   (2,0) circle (1.2pt);
\draw[fill]   (2.4,0) circle (1.2pt);
\draw[fill]   (2.8,0) circle (1.2pt);
\draw[fill]   (3.2,0) circle (1.2pt);
\draw[fill]   (3.6,0) circle (1.2pt);
\draw[fill]   (4,0) circle (1.2pt);
\draw[fill]   (4.4,0) circle (1.2pt);
\draw[fill]   (4.8,0) circle (1.2pt);
\draw[fill]   (-0.4,0) circle (1.2pt);
\draw[fill]   (-0.8,0) circle (1.2pt);
\draw[fill]   (-1.2,0) circle (1.2pt);
\draw[fill]   (-1.6,0) circle (1.2pt);
\draw[fill]   (-2,0) circle (1.2pt);
\draw[fill]   (-2.4,0) circle (1.2pt);
\draw[fill]   (-2.8,0) circle (1.2pt);
\draw[fill]   (-3.2,0) circle (1.2pt);
\draw[fill]   (-3.6,0) circle (1.2pt);
\draw[fill]   (-4,0) circle (1.2pt);
\draw[fill]   (-4.4,0) circle (1.2pt);
\draw[fill]   (-4.8,0) circle (1.2pt);
\draw[fill]   (0,0.4) circle (1.2pt);
\draw[fill]   (0.4,0.4) circle (1.2pt);
\draw[fill]   (0.8,0.4) circle (1.2pt);
\draw[fill]   (1.2,0.4) circle (1.2pt);
\draw[fill]   (1.6,0.4) circle (1.2pt);
\draw[fill]   (2,0.4) circle (1.2pt);
\draw[fill]   (2.4,0.4) circle (1.2pt);
\draw[fill]   (2.8,0.4) circle (1.2pt);
\draw[fill]   (3.2,0.4) circle (1.2pt);
\draw[fill]   (3.6,0.4) circle (1.2pt);
\draw[fill]   (4,0.4) circle (1.2pt);
\draw[fill]   (4.4,0.4) circle (1.2pt);
\draw[fill]   (4.8,0.4) circle (1.2pt);
\draw[fill]   (-0.4,0.4) circle (1.2pt);
\draw[fill]   (-0.8,0.4) circle (1.2pt);
\draw[fill]   (-1.2,0.4) circle (1.2pt);
\draw[fill]   (-1.6,0.4) circle (1.2pt);
\draw[fill]   (-2,0.4) circle (1.2pt);
\draw[fill]   (-2.4,0.4) circle (1.2pt);
\draw[fill]   (-2.8,0.4) circle (1.2pt);
\draw[fill]   (-3.2,0.4) circle (1.2pt);
\draw[fill]   (-3.6,0.4) circle (1.2pt);
\draw[fill]   (-4,0.4) circle (1.2pt);
\draw[fill]   (-4.4,0.4) circle (1.2pt);
\draw[fill]   (-4.8,0.4) circle (1.2pt);
\draw[fill]   (0,-0.4) circle (1.2pt);
\draw[fill]   (0.4,-0.4) circle (1.2pt);
\draw[fill]   (0.8,-0.4) circle (1.2pt);
\draw[fill]   (1.2,-0.4) circle (1.2pt);
\draw[fill]   (1.6,-0.4) circle (1.2pt);
\draw[fill]   (2,-0.4) circle (1.2pt);
\draw[fill]   (2.4,-0.4) circle (1.2pt);
\draw[fill]   (2.8,-0.4) circle (1.2pt);
\draw[fill]   (3.2,-0.4) circle (1.2pt);
\draw[fill]   (3.6,-0.4) circle (1.2pt);
\draw[fill]   (4,-0.4) circle (1.2pt);
\draw[fill]   (4.4,-0.4) circle (1.2pt);
\draw[fill]   (4.8,-0.4) circle (1.2pt);
\draw[fill]   (-0.4,-0.4) circle (1.2pt);
\draw[fill]   (-0.8,-0.4) circle (1.2pt);
\draw[fill]   (-1.2,-0.4) circle (1.2pt);
\draw[fill]   (-1.6,-0.4) circle (1.2pt);
\draw[fill]   (-2,-0.4) circle (1.2pt);
\draw[fill]   (-2.4,-0.4) circle (1.2pt);
\draw[fill]   (-2.8,-0.4) circle (1.2pt);
\draw[fill]   (-3.2,-0.4) circle (1.2pt);
\draw[fill]   (-3.6,-0.4) circle (1.2pt);
\draw[fill]   (-4,-0.4) circle (1.2pt);
\draw[fill]   (-4.4,-0.4) circle (1.2pt);
\draw[fill]   (-4.8,-0.4) circle (1.2pt);
\draw[fill]   (0,-0.8) circle (1.2pt);
\draw[fill]   (0.4,-0.8) circle (1.2pt);
\draw[fill]   (0.8,-0.8) circle (1.2pt);
\draw[fill]   (1.2,-0.8) circle (1.2pt);
\draw[fill]   (1.6,-0.8) circle (1.2pt);
\draw[fill]   (2,-0.8) circle (1.2pt);
\draw[fill]   (2.4,-0.8) circle (1.2pt);
\draw[fill]   (2.8,-0.8) circle (1.2pt);
\draw[fill]   (3.2,-0.8) circle (1.2pt);
\draw[fill]   (3.6,-0.8) circle (1.2pt);
\draw[fill]   (4,-0.8) circle (1.2pt);
\draw[fill]   (4.4,-0.8) circle (1.2pt);
\draw[fill]   (4.8,-0.8) circle (1.2pt);
\draw[fill]   (-0.4,-0.8) circle (1.2pt);
\draw[fill]   (-0.8,-0.8) circle (1.2pt);
\draw[fill]   (-1.2,-0.8) circle (1.2pt);
\draw[fill]   (-1.6,-0.8) circle (1.2pt);
\draw[fill]   (-2,-0.8) circle (1.2pt);
\draw[fill]   (-2.4,-0.8) circle (1.2pt);
\draw[fill]   (-2.8,-0.8) circle (1.2pt);
\draw[fill]   (-3.2,-0.8) circle (1.2pt);
\draw[fill]   (-3.6,-0.8) circle (1.2pt);
\draw[fill]   (-4,-0.8) circle (1.2pt);
\draw[fill]   (-4.4,-0.8) circle (1.2pt);
\draw[fill]   (-4.8,-0.8) circle (1.2pt);
\draw[fill]   (0,0.8) circle (1.2pt);
\draw[fill]   (0.4,0.8) circle (1.2pt);
\draw[fill]   (0.8,0.8) circle (1.2pt);
\draw[fill]   (1.2,0.8) circle (1.2pt);
\draw[fill]   (1.6,0.8) circle (1.2pt);
\draw[fill]   (2,0.8) circle (1.2pt);
\draw[fill]   (2.4,0.8) circle (1.2pt);
\draw[fill]   (2.8,0.8) circle (1.2pt);
\draw[fill]   (3.2,0.8) circle (1.2pt);
\draw[fill]   (3.6,0.8) circle (1.2pt);
\draw[fill]   (4,0.8) circle (1.2pt);
\draw[fill]   (4.4,0.8) circle (1.2pt);
\draw[fill]   (4.8,0.8) circle (1.2pt);
\draw[fill]   (-0.4,0.8) circle (1.2pt);
\draw[fill]   (-0.8,0.8) circle (1.2pt);
\draw[fill]   (-1.2,0.8) circle (1.2pt);
\draw[fill]   (-1.6,0.8) circle (1.2pt);
\draw[fill]   (-2,0.8) circle (1.2pt);
\draw[fill]   (-2.4,0.8) circle (1.2pt);
\draw[fill]   (-2.8,0.8) circle (1.2pt);
\draw[fill]   (-3.2,0.8) circle (1.2pt);
\draw[fill]   (-3.6,0.8) circle (1.2pt);
\draw[fill]   (-4,0.8) circle (1.2pt);
\draw[fill]   (-4.4,0.8) circle (1.2pt);
\draw[fill]   (-4.8,0.8) circle (1.2pt);
\draw[fill]   (0,1.2) circle (1.2pt);
\draw[fill]   (0.4,1.2) circle (1.2pt);
\draw[fill]   (0.8,1.2) circle (1.2pt);
\draw[fill]   (1.2,1.2) circle (1.2pt);
\draw[fill]   (1.6,1.2) circle (1.2pt);
\draw[fill]   (2,1.2) circle (1.2pt);
\draw[fill]   (2.4,1.2) circle (1.2pt);
\draw[fill]   (2.8,1.2) circle (1.2pt);
\draw[fill]   (3.2,1.2) circle (1.2pt);
\draw[fill]   (3.6,1.2) circle (1.2pt);
\draw[fill]   (4,1.2) circle (1.2pt);
\draw[fill]   (4.4,1.2) circle (1.2pt);
\draw[fill]   (4.8,1.2) circle (1.2pt);
\draw[fill]   (-0.4,1.2) circle (1.2pt);
\draw[fill]   (-0.8,1.2) circle (1.2pt);
\draw[fill]   (-1.2,1.2) circle (1.2pt);
\draw[fill]   (-1.6,1.2) circle (1.2pt);
\draw[fill]   (-2,1.2) circle (1.2pt);
\draw[fill]   (-2.4,1.2) circle (1.2pt);
\draw[fill]   (-2.8,1.2) circle (1.2pt);
\draw[fill]   (-3.2,1.2) circle (1.2pt);
\draw[fill]   (-3.6,1.2) circle (1.2pt);
\draw[fill]   (-4,1.2) circle (1.2pt);
\draw[fill]   (-4.4,1.2) circle (1.2pt);
\draw[fill]   (-4.8,1.2) circle (1.2pt);
\draw[fill]   (0,-1.2) circle (1.2pt);
\draw[fill]   (0.4,-1.2) circle (1.2pt);
\draw[fill]   (0.8,-1.2) circle (1.2pt);
\draw[fill]   (1.2,-1.2) circle (1.2pt);
\draw[fill]   (1.6,-1.2) circle (1.2pt);
\draw[fill]   (2,-1.2) circle (1.2pt);
\draw[fill]   (2.4,-1.2) circle (1.2pt);
\draw[fill]   (2.8,-1.2) circle (1.2pt);
\draw[fill]   (3.2,-1.2) circle (1.2pt);
\draw[fill]   (3.6,-1.2) circle (1.2pt);
\draw[fill]   (4,-1.2) circle (1.2pt);
\draw[fill]   (4.4,-1.2) circle (1.2pt);
\draw[fill]   (4.8,-1.2) circle (1.2pt);
\draw[fill]   (-0.4,-1.2) circle (1.2pt);
\draw[fill]   (-0.8,-1.2) circle (1.2pt);
\draw[fill]   (-1.2,-1.2) circle (1.2pt);
\draw[fill]   (-1.6,-1.2) circle (1.2pt);
\draw[fill]   (-2,-1.2) circle (1.2pt);
\draw[fill]   (-2.4,-1.2) circle (1.2pt);
\draw[fill]   (-2.8,-1.2) circle (1.2pt);
\draw[fill]   (-3.2,-1.2) circle (1.2pt);
\draw[fill]   (-3.6,-1.2) circle (1.2pt);
\draw[fill]   (-4,-1.2) circle (1.2pt);
\draw[fill]   (-4.4,-1.2) circle (1.2pt);
\draw[fill]   (-4.8,-1.2) circle (1.2pt);
\draw (-6.25,1.8) node[anchor=north west] {$T(R^*)$};
\draw (-6.3,-.6) node[anchor=north west] {$B(R^*)$};
\draw (5.5,.5) node[anchor=north west] {\Large $R^*$};
\draw [line width=1pt] (-.8,1.2) -- (-0.8,.8);
\draw [line width=1pt] (-1.2,.8) -- (-0.8,.8);
\draw [line width=1pt] (-1.2,.8) -- (-1.2,.4);
\draw [line width=1pt] (-.8,.4) -- (-1.2,.4);
\draw [line width=1pt] (-.8,.4) -- (-.4,.4);
\draw [line width=1pt] (0,.4) -- (-.4,.4);
\draw [line width=1pt] (0,.4) -- (0,0);
\draw [line width=1pt] (0.4,0) -- (0,0);
\draw [line width=1pt] (0.4,0) -- (0.8,0);
\draw [line width=1pt] (0.8,-0.4) -- (0.8,0);
\draw [line width=1pt] (0.8,-0.4) -- (0.8,-.8);
\draw [line width=1pt] (1.2,-0.8) -- (0.8,-.8);
\draw [line width=1pt] (1.2,-0.8) -- (1.2,-.4);
\draw [line width=1pt] (1.6,-0.4) -- (1.2,-.4);
\draw [line width=1pt] (1.6,-0.4) -- (1.6,-.8);
\draw [line width=1pt] (1.6,-1.2) -- (1.6,-.8);
\end{tikzpicture}
\caption{A configuration $\om^*$ is good transversally if it has a path of open edges connecting $T(R^*)$ with $B(R^*)$.} \label{figure2}
\end{figure}

We say that a configuration $\om\in\llav{0,1}^{E(R)}$ is good lengthwise if its dual configuration $\om^*\in\llav{0,1}^{E(R)^*}$ does not have any path of open edges connecting $T(R^*)$ with $B(R^*)$; in this case, we say $\om^*$ is bad transversally.
Observe that a sufficient condition for a configuration $\om\in\llav{0,1}^{E(\Delta_K)}$ to have a circuit is that, for every $1\le i\le 4$, the projection $\om_{E(R_i)}$ is good lengthwise.
We have
\begin{equation}
B_\rho (\{\om\in \llav{0,1}^{E(\Delta_K)}:\om\textnormal{ has a circuit}\})
 \geq 
\pare{ B_\rho (\{\om\in \llav{0,1}^{E(R)}:\om \textnormal{ is good lengthwise}\})}^4;
\end{equation}
in the last inequality, we used the fact that the event $\llav{\om\in\llav{0,1}^{E(R)}:\omega \textnormal{ is good lengthwise}}$ is increasing and that the Bernoulli probability satisfies the FKG property;
see \cite{grimett}.
To estimate the probability of the last set, we consider its complement:
\begin{equation}
\llav{\om\in \llav{0,1}^{E(R)}:\om^* \textnormal{ is good transversaly}}.
\end{equation}
Observe that, if $\om^*$ is good transversally, there exists a self-avoiding path $\gamma$ of open edges starting in $B(R^*)$ and such that $\abs{\gamma}=\lfloor\sqrt{K}\rfloor$, where $\lfloor\sqrt{K}\rfloor$ denotes the integer part of $\sqrt{K}$ and $\abs{\gamma}$ the number of edges of $\gamma$.
Then
\begin{align}
&B_\rho(\{\om\in \llav{0,1}^{E(R)}:\om^* \textnormal{ is good transversaly}\})\nonumber
\\[0.3cm]
& \hspace{1cm}
\leq  B_\rho\Bigg(\bigcup_{x^*\in B(R^*)} \ \bigcup_{\substack{ \gamma\textnormal{ starting at }x^* \\[0.05cm] \abs{\gamma}=\lfloor\sqrt{ K}\rfloor }}\llav{\om\in \llav{0,1}^{E(R)}:\om^*_{e^*}=1 \textnormal{ for every }e^*\in\gamma   }\Bigg)
\\[0.3cm]
& \hspace{1cm}\leq
\sum_{x^*\in B(R^*)} \ \sum_{\substack{ \gamma\textnormal{ starting at }x^* \\[0.05cm] \abs{\gamma}=\lfloor\sqrt{ K }\rfloor }}
B_\rho(\{\om\in \llav{0,1}^{E(R)}:\om^*_{e^*}=1 \textnormal{ for every }e^*\in\gamma   \})
\\[0.3cm]
& \hspace{1cm}\leq
K 3^{\lfloor\sqrt{K}\rfloor}
(1-\rho)^{\lfloor\sqrt{K}\rfloor}.
\end{align}
We conclude that
\begin{align}
&B_\rho(\{\om\in \llav{0,1}^{E(C)}:\om\textnormal{ does not have a circuit}\})
\leq
1-\pare{1-K 3^{\lfloor\sqrt{K}\rfloor}
(1-\rho)^{\lfloor\sqrt{K}\rfloor}}^4
\\[0.3cm]
& \hspace{1cm}
\leq 1-\pare{1-8K\corch{3(1-\rho)}^{\lfloor\sqrt{K}\rfloor}}=
8K\corch{3(1-\rho)}^{\lfloor\sqrt{K}\rfloor}.
\end{align}
Coming back to \eqref{rr44}, we obtain the upper bound
$\corch{C_1 \pare{8K\corch{3(1-\rho)}^{\lfloor\sqrt{K}\rfloor}}^\delta}^{N_{K,L}}$.
Condition $\beta>\log\sqrt 5$ is equivalent to $3(1-\rho)<1$.
Take $K$ large enough to satisfy $\pare{8K\corch{3(1-\rho)}^{\lfloor\sqrt{K}\rfloor}}^\delta<1$ to conclude. \qed

\bigskip

\subsection{Proof of \eqref{lb}.}\label{denominator}

Given $\eta\in I_{|C_{l}|}$, if $|\eta| < m_\beta$
we choose $p<1$ such that $\eta=pm_{\beta}-(1-p)m_{\beta}$.
Supposing that $\Delta_L=[0,L)^d$ we split it as: $\Delta_L=\Delta_L^+\cup\Delta\cup \Delta_L^-$, where
\begin{equation}\label{Delta}
\Delta \defi \{
x=(x_1,\ldots, x_d) \in \Delta_L:\, |x_1-p L|\leq c^*
\}
\end{equation}
for an appropriate $c^*>0$ to be chosen next.
The set $\Delta_L^+$ (respectively $\Delta_L^-$) is the part of the domain corresponding to
smaller (respectively larger) values of $x_1$.
With this definition, letting $\tau_1=+$ and $\tau_2=-$,
there is a $c>0$ such that
\begin{equation}\label{muplus}
\mu^{\nn}_{\Delta_L^{\tau_i},0,\tau_{i}}\Big(
|\sum_{x\in \Delta_L^{\tau_i}}(\sigma(x)- \tau_i m_{\beta})|\leq c |\Delta_L^{\tau_i}|^{1/2}
\Big)
> \frac 12.
\end{equation}
We express the set in \eqref{muplus} by the union (over all possible magnetizations $m_i$, $i=1,2$,
with $|m_i-\tau_i m_{\beta}|\leq c |\Delta_L^{\tau_i}|^{-1/2}$) of the set 
$\{\sum_{x\in \Delta_L^{\tau_i}}\sigma(x)=m_i | \Delta_L^{\tau_i}|\}$. Then,
it follows that there are two values $m^{*}_i$, $i=1,2$, of the magnetization
with the above constraint so that
\begin{equation}\label{question}
\mu^{\nn}_{\Delta_L^{\tau_i},0,\tau_{i}}\Big(
\sum_{x\in \Delta_L^{\tau_i}}\sigma(x)=m^{*}_i | \Delta_L^{\tau_i}| 
\Big)
\geq \frac 12 \frac{2}{c | \Delta_L^{\tau_i}|^{1/2}}.
\end{equation}
With this, we define the set
\begin{equation}\label{G}
\mathcal G_\eta=\{
\sigma:\, \sum_{x\in \Delta_L^{\tau_i}}\sigma (x)=m^{*}_i |\Delta_L |, \, i=1,2,\,
\sum_{x\in \Delta}\sigma_{\Lambda}(x)=b|\Delta|
\},
\end{equation}
with $b$ such that $m^{*}_1| \Delta_L^+ |+m^{*}_2 | \Delta_L^- | + b|\Delta|= \eta |\Delta_L |$.
It is easy to see that such a $b$ exists for a $c^*$ in \eqref{Delta} large enough.

Thus, to get a lower bound to the left hand side of \eqref{lb} we restrict to $\mathcal G_\eta$.
As a consequence, in each subdomain the 
corresponding probabilities can be bounded as in \eqref{question}
and we are left with only
the boundary terms, which are of the order $L^{d-1}$
(for a box $\Delta_L=L^d$). We have
\begin{eqnarray}\label{tlb}
\mu^{\nn}_{\Delta_L, 0}(\{m_{\Delta_{L}}=\eta\}) & \geq & \mu^{\nn}_{\Delta_L, 0}(\mathcal G_\eta) 
\nonumber\\
& \geq &
\frac{1}{Z^{\nn}_{\Delta_L,0}}
e^{-2\beta J (2d+2) L^{d-1}}\times \nonumber\\
&&
\sum_{\sigma_{\Delta_L^+},\, \sigma_{\Delta_L^-},\, \sigma_{\Delta}}
\mathbf 1_{\mathcal G_\eta}
\,
e^{-\beta H_{\Delta_L^+}(\sigma_{\Delta_L^+} | \mathbf 1_{(\Delta_L^+)^c})}
e^{-\beta H_{\Delta_L^-}(\sigma_{\Delta_L^-} | \mathbf 1_{(\Delta_L^-)^c})}
e^{-\beta H_{\Delta}(\sigma_{\Delta})}
\nonumber\\
& \geq & 
\frac{1}{Z^{\nn}_{\Delta_L,0}}
e^{-2\beta J (2d+2) L^{d-1}}
e^{-\beta 2d J |\Delta|}
\prod_{i=1,2}\Big(\frac{1}{c |\Delta_L^{\tau_i}|^{1/2}}Z^{\nn}_{\Delta_L^{\tau_i},0,\tau_{i}}\Big).
\end{eqnarray}
Then we can easily conclude since
\begin{eqnarray}\label{conclude}
Z^{\nn}_{\Delta_L^+,0,+} Z^{\nn}_{\Delta_L^-,0,-} & \geq &
Z^{\nn}_{\Delta_L^+,0,+} Z^{\nn}_{\Delta_L^-,0,-} Z^{\nn}_{\Delta,0}
2^{-\Delta} e^{-\beta d J |\Delta|}\nonumber\\
& \geq &
Z^{\nn}_{\Delta_L,0} 
e^{-2\beta J (2d+2) L^{d-1}}
2^{-\Delta} e^{-\beta d J |\Delta|}.
\end{eqnarray}

If $|\eta|>m_\beta$, then we choose $h \defi h(\eta)$ as in discussion following
\eqref{mbeta}
and obtain:
\begin{equation}\label{foreta}
\mu^{\nn}_{\Delta_L,h(\eta)}\Big(|
\sum_{x\in \Delta_L} \sigma(x)
-\eta|\geq c |\Delta_L|^{1/2}
\Big)\leq \frac{1}{c^2 |\Delta_L|}\sum_{x,y}\E_{\mu^{\nn}_{\Delta_L,h(\eta)}}
[|(\sigma(x)-\eta)(\sigma(y)-\eta)|]
\leq \frac 12
\end{equation}
for an appropriate choice of $c$.
Following the steps above, \eqref{foreta} implies \eqref{lb} for the case $|\eta|>m_\beta$
and concludes the proof. \qed

\bigskip

{\bf Acknowledgements.}
It is a great pleasure to thank Errico Presutti for suggesting us the problem and for his continuous advising.
We also acknowledge fruitful suggestions from Roman Koteck\'y at an earlier version of the paper.
This work was initiated when NSL and DT were visiting GSSI. 
NSL also acknowledges the hospitality of the University of Sussex.
The visit of NSL to GSSI was partially supported by grant 2010-MINCYT-NIO Interacting stochastic systems: fluctuations, hydrodynamics, scaling limits;
the visit of NSL to the University of Sussex was partially supported by 
grant 14-­STIC-­03 Evaluation and Optimal Control of High-­dimensional stochastic network - ECHOS.
NSL was also supported by a scholarship from CONICET-Argentina during the course of his PhD.

\end{document}